\documentclass{amsart}

\usepackage{amsmath}
\usepackage{bbm}
\usepackage[utf8x]{inputenc}
\usepackage{latexsym}
\usepackage{xcolor}
\usepackage{tikz}
\usetikzlibrary{arrows}
\usetikzlibrary{patterns}
\usetikzlibrary{shapes}
\usepackage{mathrsfs}
\usepackage{verbatim}
\usepackage{bbm}

\usepackage[mathcal]{euscript}

\newcommand{\R}{\mathbb{R}}
\newcommand{\N}{\mathbb{N}}

\newcommand{\eexp}{\overrightarrow{\exp}}
\newcommand{\Lp}[2]{L^{#1}(#2)}
\newcommand{\gut}[2]{g_{#1}^{#2}}
\newcommand{\dgut}[2]{\dot g_{#1}^{#2}}
\newcommand{\Put}[2]{P_{#1}^{#2}}
\newcommand{\bt}{{\overline{t}}}

\usepackage[usenames,dvipsnames]{pstricks}
\usepackage{epsfig}
\usepackage{pst-grad} 
\usepackage{pst-plot} 

\newtheorem{thm}{Theorem}
\newtheorem{lemma}[thm]{Lemma}
\newtheorem{cor}[thm]{Corollary}
\newtheorem{prop}[thm]{Proposition}

\newtheorem*{thmi}{Theorem}
\newtheorem{defi}{Definition}

\theoremstyle{remark}
\newtheorem{remark}[]{Remark}

\newcommand{\be}{\begin{equation}}
\newcommand{\ee}{\end{equation}}

\usepackage{mathtools} 
\mathtoolsset{showonlyrefs,showmanualtags} 

\title{Homotopically invisibile singular curves}
\author{Andrei A. Agrachev}
\address{SISSA (Trieste) and Steklov Institute (Moscow)}
\email{agrachev@sissa.it}

\author{Francesco Boarotto}
\address{SISSA (Trieste)}
\email{fboaro@sissa.it}

\author{Antonio Lerario}\address{SISSA (Trieste)}
\email{lerario@sissa.it}

\begin{document}
		
	\begin{abstract}
		Given a smooth manifold $M$ and a totally nonholonomic distribution $\Delta\subset TM$ of rank $d$, we study the effect of singular curves on the topology of the space of horizontal paths joining two points on $M$. Singular curves are critical points of the endpoint map $F:\gamma\mapsto\gamma(1)$ defined on the space $\Omega$ of horizontal paths starting at a fixed point $x$. We consider a subriemannian energy $J:\Omega(y)\to\mathbb R$, where $\Omega(y)=F^{-1}(y)$ is the space of horizontal paths connecting $x$ with $y$, and study those singular paths that do not influence the homotopy type of the Lebesgue sets $\{\gamma\in\Omega(y)\,|\,J(\gamma)\le E\}$. We call them {\it homotopically invisible}. It turns out that for $d\geq 3$ generic subriemannian structures in the sense of \cite{CJT} have only homotopically invisible singular curves.
		Our results can be seen as a first step for developing the calculus of variations on the singular space of horizontal curves (in this direction we prove a subriemannian Minimax principle and discuss some applications).
	\end{abstract}
	
	\maketitle
	
	\section{Introduction}\label{sec:Introduction}
	\subsection{Horizontal path spaces and singular curves}Let $M$ be a smooth manifold of dimension $m$  and $\Delta\subset TM$ be a smooth, totally nonholonomic distribution of rank $d$.
Given a point $x\in M$ (which we will assume fixed once and for all) the horizontal path space $\Omega$ of \emph{admissible} curves is defined\footnote{This definition requires the choice of an inner product in each fiber (a subriemannian structure on $\Delta$) in order to integrate the square of the norm of $\dot\gamma$, but the fact of being \emph{integrable} is independent of the chosen structure (we refer the reader to \cite{AgrachevBarilariBoscain, Montgomery} for more details).}  by:
\be \Omega=\{\gamma:[0,1]\to M\,|\, \gamma(0)=x, \,\textrm{$\gamma$ is absolutely continuous, $\dot{\gamma}\in \Delta$ a.e. and is $L^2$-integrable}\}.\ee
The $W^{1,2}$ topology endows $\Omega$ with a Hilbert manifold structure, locally modeled on $L^2(I, \R^d)$. 

The \emph{endpoint map} is the smooth map assigning to each curve its final point:
\be F:\Omega\to M,\quad F(\gamma)=\gamma(1).\ee
A \emph{singular} curve is a critical point of $F$.
Singular curves are at the core of nonholonomic geometry, but some natural questions about these curves remain open. 
In fact even the existence of a point $y \in M$ not joined by singular curves sorting from $x$ is still an open problem (the ``subriemannian Sard's conjecture'').

Given $y\in M$ we will denote by $\Omega(y)$ the set of horizontal curves joining $x$ and $y$:
\be \Omega(y)=F^{-1}(y), \quad y\in M.\ee
If $y$ is a regular value of $F$, then there are no singular curves between $x$ and $y$ and the space $\Omega(y)$ is a smooth Hilbert manifold. As we just said, the absence of singular curves can not be granted in general and the space $\Omega(y)$ might be very singular. Despite this fact, from a homotopy theory point of view $\Omega(y)$ still turns out to be a ``nice'' space.
\begin{center}
\begin{figure}
\scalebox{1} 
{
\begin{pspicture}(0,-2.6080468)(8.461895,2.6080468)
\psline[linewidth=0.04cm](0.0,-2.0303516)(8.24,-2.0103517)
\psdots[dotsize=0.12](1.8,-2.0303516)
\psdots[dotsize=0.12](5.38,-2.0303516)
\psbezier[linewidth=0.04](2.2875562,2.0696485)(2.16,0.70964843)(1.22,1.0000654)(1.24,0.12964843)(1.26,-0.7407686)(2.08,0.12964843)(2.46,-1.5830531)
\psbezier[linewidth=0.04](5.6,2.0496485)(5.82,0.53143615)(5.18402,-0.1021697)(4.78,-0.08353592)(4.37598,-0.06490214)(5.72,-0.025882289)(5.78,-1.3903515)
\usefont{T1}{ptm}{m}{n}
\rput(1.791455,-2.3453515){$y_1$}
\usefont{T1}{ptm}{m}{n}
\rput(5.411455,-2.3853517){$y_2$}
\usefont{T1}{ptm}{m}{n}
\rput(1.8214551,2.4146485){$\Omega(y_1)$}
\usefont{T1}{ptm}{m}{n}
\rput(5.401455,2.3946486){$\Omega(y_2)$}
\usefont{T1}{ptm}{m}{n}
\rput(7.681455,-1.6453515){$M$}
\psline[linewidth=0.04cm,arrowsize=0.05291667cm 2.0,arrowlength=1.4,arrowinset=0.4]{->}(7.66,1.9896485)(7.66,-0.19035156)
\usefont{T1}{ptm}{m}{n}
\rput(8.141455,1.0346484){$F$}
\end{pspicture}
}

\caption{A simplified picture of two fibers of the enpdoint map $F:\Omega\to M$. The path space $\Omega(y_2)$ is singular, but it is homotopy equivalent to $\Omega(y_1)$.}\label{fig:endpoint}
\end{figure}
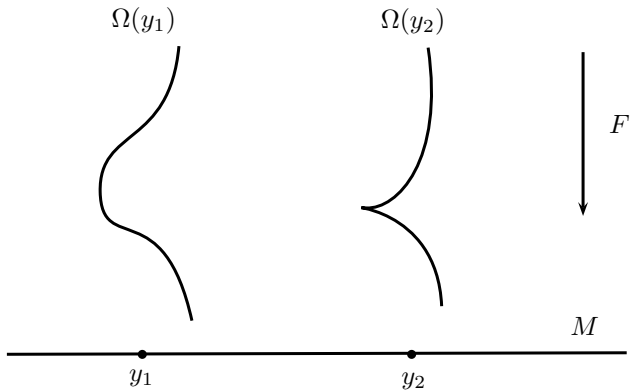
\end{center}
\begin{thmi}[Theorem 5, \cite{BoarottoLerario}] For every $y\in M$ the space $\Omega(y)$ has the homotopy type of a CW-complex and its inclusion in the standard path space (i.e. the space of $W^{1,2}$ curves on $M$ with no nonholonomic constraint on their velocity) is a homotopy equivalence. In particular all the spaces $\Omega(y)$ (regardless $y\in M$ being a regular value of $F$ or not) have the same homotopy type.
\end{thmi}

Thus, globally the homotopy type of $\Omega(y)$, $y\in M$, is not influenced by the fact of being singular. In fact all fibers of the endpoint map can be continuously deformed one into another, but if an additional function $J:\Omega\to \R$ (an energy functional) is considered, during the deformation we cannot preserve the Lebesgue sets of $J$.
\subsection{Soft singular curves and homotopically invisible curves} In this framework, the main interest of calculus of variations is to determine the existence of critical points of a functional:
\be J:\Omega(y)\to \R.\ee
This is often done using the homotopy of the space $\Omega(y)$ (which in this case is known by the previous theorem) to force the existence of critical points.

In our case the role of $J$ will be played by a \emph{subriemannian Energy}. In other words, we fix an inner product on $\Delta$ smoothly depending on the base-point and define for $\gamma \in \Omega$:
\be J(\gamma)=\frac{1}{2}\int_{I}|\dot{\gamma}(t)|^2dt.\ee
If $y$ is a regular value of $F$, then $\Omega(y)$ is smooth and a critical point of $J$ is a curve $\gamma$ for which there exists a nonzero $\lambda\in T^*M$ such that:
\be \lambda d_{\gamma }F=d_\gamma J.\ee
In the language of subriemannian geometry these curves are called \emph{normal geodesics}; their short segments are length minimizers for the corresponding Carnot-Caratheodory distance on $M$ (not all length minimizers are normal geodesics though).

In the spirit of Morse theory, when $\Omega(y)$ is smooth, normal geodesics with energy in $[E_1, E_2]$ are precisely the obstruction to deform the Lebesgue set $\{J\leq E_2\}$ to $\{J\leq E_1\}$ following the gradient flow of $-J$ (if $y$ is a regular value of $F$, the function $J$ satisfies the Palais-Smale condition \cite[Proposition 10]{BoarottoLerario} and the classical theory can be used; we refer the reader to \cite{BoarottoLerario}).

If $\Omega(y)$ is singular, even if there are no normal geodesics with energy in $[E_1, E_2]$, the same deformation is in general not possible (in fact it is not even clear what a gradient flow should be on this singular space).  Neverthless, for a generic (in the sense of \cite{CJT}) subriemannian structure of rank $d\geq 3$ the absence of normal geodesics with energy in $[E_1, E_2]$ is enough to guarantee a \emph{weak} deformation. Denoting by $\Omega(y)^E$ the set $\{\gamma\in \Omega(y)\,|\,J(\gamma)\leq E\}$, the main result of this paper implies the following theorem.
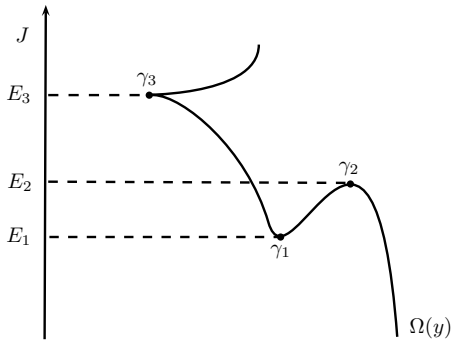
\begin{figure}
\scalebox{0.8} 
{
\begin{pspicture}(0,-2.8176954)(8.28291,2.8)
\psbezier[linewidth=0.04](4.4210157,2.12)(4.4210157,1.2)(2.2010157,1.3)(2.6810157,1.28)(3.1610155,1.26)(4.2154155,0.4)(4.5810156,-0.84)(4.9466157,-2.08)(6.3319345,2.4)(6.7210155,-2.74)
\psdots[dotsize=0.12](2.6010156,1.28)
\psdots[dotsize=0.12](4.7810154,-1.08)
\psdots[dotsize=0.12](5.9410157,-0.2)
\psline[linewidth=0.04cm,arrowsize=0.05291667cm 2.0,arrowlength=1.4,arrowinset=0.4]{->}(0.8610156,-2.78)(0.8810156,2.78)
\psline[linewidth=0.04cm,linestyle=dashed,dash=0.16cm 0.16cm](2.4410157,1.28)(0.9210156,1.28)
\psline[linewidth=0.04cm,linestyle=dashed,dash=0.16cm 0.16cm](5.861016,-0.18)(0.9210156,-0.16)
\psline[linewidth=0.04cm,linestyle=dashed,dash=0.16cm 0.16cm](4.821016,-1.08)(0.9410156,-1.08)
\usefont{T1}{ptm}{m}{n}
\rput(0.4924707,2.285){$J$}
\usefont{T1}{ptm}{m}{n}
\rput(2.5624707,1.545){$\gamma_3$}
\usefont{T1}{ptm}{m}{n}
\rput(5.9224706,0.065){$\gamma_2$}
\usefont{T1}{ptm}{m}{n}
\rput(4.7824707,-1.315){$\gamma_1$}
\usefont{T1}{ptm}{m}{n}
\rput(0.4324707,1.285){$E_3$}
\usefont{T1}{ptm}{m}{n}
\rput(0.4524707,-0.175){$E_2$}
\usefont{T1}{ptm}{m}{n}
\rput(0.4324707,-1.075){$E_1$}
\usefont{T1}{ptm}{m}{n}
\rput(7.2824707,-2.595){$\Omega(y)$}
\end{pspicture}
}

\caption{A simplified picture of a fiber $\Omega(y)$ and the Energy function (in the vertical direction). The curves $\gamma_1$ and $\gamma_2$ represent normal geodesics (they are obstructions to deform the Lebesgue sets). The curve $\gamma_3$ represents a soft singular curve: any cycle $X\subset \Omega(y)^{E_3}$ can be deformed a bit below the level $E_3.$  }\label{fig:critical}

\end{figure}

\begin{thm}[subriemannian Deformation Lemma]\label{thm:deformationintro}For a generic subriemannian structure of rank $d\geq 3$ on $M$ the following statement is true.  Assume there are no normal geodesics in $\Omega(y)$ with energy in $[E_1, E_2]$. Then for every compact manifold $X$, every continuous map $f:X\to \Omega(y)^{E_2}$ and every $\epsilon>0$ there exists a homotopy $f_t:X\to \Omega(y)^{E_2}$ such that:
\be f_0=f\quad \textrm{and}\quad f_1(X)\subset \Omega(y)^{E_1+\epsilon}.\ee
In particular $\Omega(y)^{E_2}$ and $\Omega(y)^{E_1+\epsilon}$ are weakly homotopy equivalent.
\end{thm}

The previous theorem says that, in the generic situation, singular curves with energy $E_1\leq J\leq E_2$ aren't obstacles for the deformation of continuous maps (see Figure \ref{fig:critical}). They are ``homotopically invisible''.

The technical condition that we need in order to guarantee the conclusion of the statement is that all singular curves with $J\leq E_2$ satisfy the following three properties: (a) they have corank one, (b) they are not Goh and (c) they are strictly abnormal; when $d\geq 3$ subriemannian structures whose all singular curves satisfy these conditions form an  open dense set in the $C^{\infty}$-Whitney topology by a result of Y. Chitour, F. Jean, and E. Tr\'elat \cite{CJT} (see Theorem \ref{thm:CJT} below for a more precise statement).

We call a singular curve satisfying conditions (a), (b) and (c) a \emph{soft} singular curve; 
the message of this paper is to show that 
soft singular curves are homotopically invisible.

\subsection{The calculus of variations on the horizontal path space}Theorem \ref{thm:deformationintro} above can be considered as the starting point for the variational analysis on the space of horizontal curves. As a matter of fact a \emph{strong} deformation retraction of Lebesgue sets is not needed even in the classical theory: all one needs is to be able to deform continuous maps (and more generally singular chains representing homology classes: we will prove that this is possible in Section \ref{sec:applications}). Neither one needs to deform up to the level $E_1$ included: the possibility of getting arbitrarily close (with the Energy) is still enough to use the results and predict the existence of critical points (i.e. normal geodesics).

If two functions $f, g:X\to Y$  between topological spaces are homotopic, we will write $f\sim g.$ The following statement is a subriemannian version of the Minimax principle.

\begin{thm}[subriemannian Minimax principle]\label{thm:minmax} 
For a generic subriemannian structure of rank $d\geq 3$ on $M$ the following statement is true. 
Let $X$ be a compact manifold and $f:X\to \Omega(y)$ be a continuous map. Consider:
\be c=\inf_{g\sim f}\sup_{\theta \in X} J(g(\theta)).\ee
Then for every $\epsilon>0$ there exists a normal geodesic $\gamma_\epsilon$ such that:
\be c-\epsilon\leq J(\gamma_\epsilon)\leq c+\epsilon.\ee

\end{thm}

It should be clear at this point that in principle, using Theorem \ref{thm:deformationintro}, the powerful machinery of classical critical point theory can be adapted to the generic subriemannian framework -- to mention a specific example, we will show how to prove an analogue of Serre's theorem on the existence of infinitely many normal geodesics joining two points on a compact manifold (Corollary \ref{cor:serre} below).

\begin{remark}The study of the space of maps with some restrictions on their differential goes back to the works on immersions of S. Smale \cite{Smale1}; the case of trajectories of affine control systems were studied first by A. V.  Sarychev \cite{Sarychev} for the uniform convergence topology, by J. Dominy and H. Rabitz \cite{dynamic} for the $W^{1,1}$ topology and by the last two authors of this paper for the $W^{1,p}, p>1$ topology \cite{BoarottoLerario}. The case of \emph{closed} $W^{1,2}$ curves on nonholonomic distribution has been addressed by the last author and A. Mondino on \cite{LerarioMondino}. A subriemannian version of Serre's theorem was proved by the last two authors of the current paper under the assumption that $y\in M$ is a regular value of $F$ \cite{BoarottoLerario}. The paper \cite{LerarioMondino} also contains various related results on variational problems on the closed horizontal loop space in the contact case.
\end{remark}

\subsection{Structure of the paper}The paper is organized as follows. In Section \ref{sec:preliminaries} we recall the main definitions and properties of the objects we use. We note that Section \ref{sec:global} contains an interesting tool (the ``global chart'') which allows to switch from the space of curves to the space of controls (where the objects can be handled easier; we will switch back to the language of curves in the last section). Section \ref{sec:cross} contains the main technical idea, which is the construction of a \emph{cross-section} (a map ``parametrizing'' all possible values) for the pair $(F, J)$ near a singular soft curve. This is used in Section \ref{sec:implicit} to prove a quantitative \emph{non-smooth} implicit function theorem (Theorem \ref{thm:implicit}) near a soft singular curve. In Section \ref{sec:regular} we introduce the technical ingredients for handling the deformation on a singular space but away from singular curves. The subriemannian Deformation Lemma (Theorem \ref{thm:deformation}) is proved in Section \ref{sec:deformation}. Indeed the invisibility of soft singular curves is a corollary of the Serre fibration property (Theorem \ref{thm:serre}). In Section \ref{sec:applications} we discuss some applications and prove the subriemannian Minimax principle (Theorem \ref{thm:minmax}) and Serre's theorem on the existence of infinitely many normal geodesics on a compact subriemannian manifold (Corollary \ref{cor:serre}).
	
	\section{Preliminaries}\label{sec:preliminaries}
\subsection{The endpoint map}Recall that we are considering a smooth totally nonholonmic distribution $\Delta\subset TM$ and that we have denoted by $\Omega$ the space:
\be \Omega=\{\gamma:[0,1]\to M\,|\, \gamma(0)=x, \,\textrm{$\gamma$ is absolutely continuous, $\dot{\gamma}\in \Delta$ a.e. and is $L^2$-integrable}\}.\ee
This space endowed with the $W^{1,2}$ has a Hilbert manifold structure, locally modeled on $L^2(I, \R^d)$; we will call this topology the \emph{strong} topology -- the \emph{weak} topology can also be considered on $\Omega$ (see \cite{BoarottoLerario, LerarioMondino, Montgomery} for more details on these topologies).

The \emph{endpoint map} is the map that gives the final point of each curve:
\be F:\Omega\to M, \quad F(\gamma)=\gamma(1).\ee If $y\in M$, we indicate by: \be\Omega(y)=F^{-1}(y)=\{\textrm{horizontal curves joining $x$ and $y$}\}.\ee
We recall some properties of the Endpoint map, see \cite[Theorem 19 and Theorem 23]{BoarottoLerario}.
\begin{prop}\label{propo:contendp}The endpoint map $F:\Omega\to M$ is smooth (with respect to the Hilbert manifold structure on $\Omega$). Moreover if $\gamma_n\rightharpoonup\gamma$ weakly, then $\gamma_n\to \gamma$ uniformly (in particular $F$ is continuous for the weak topology) and $d_{\gamma_n}F\to d_{\gamma}F$ in the operator norm.\end{prop}

An horizontal (admissible) curve $\gamma$ is said to be \emph{singular} if it is a critical point of the endpoint map. The codimension of the singularity is called the \emph{corank} of the singular curve.

\subsection{Abnormal extremals}
Let us consider the cotangent bundle $T^*M$, and let us fix on it the standard symplectic structure $\omega$. It is possible to find a distinguished subspace $\Delta^\perp$ within $T^*M$, called the annihilator of the distribution $\Delta$, and accordingly we may restrict $\omega$ to a two-form $\overline{\omega}$ on $\Delta^\perp$, which may now have characteristics.

An absolutely continuous curve $\lambda:[0,1]\to \Delta^\perp$ is an \emph{abnormal} extremal of $\Delta$ if $\dot{\lambda}(t)\in\ker\overline{\omega}(\lambda(t))$ for any $t\in[0,1]$.

There exists a remarkable connection between singular curves (which are defined on the manifold $M$), and abnormal extremals (on $T^*M$), that is that an horizontal curve $\gamma$ is singular if and only if it is the projection of an abnormal extremal $\lambda$ on the cotangent space. If this is the case, we say that $\lambda$ is an \emph{abnormal extremal lift} of $\gamma$.

\subsection{The Energy and the extended endpoint map} A subriemannian structure
on $(M,\Delta)$ is a riemannian metric on $\Delta$, i.e. a scalar product on $\Delta$ which smoothly depends on
the base point. If $\Delta$ is endowed with a subriemannian structure, we can define the \emph{Energy} of horizontal paths: \be J:\Omega\to\R,\quad J(\gamma)=\int_0^1 |\dot{\gamma}(t)|^2dt.\ee

The Energy is a smooth (hence continuous) map on $\Omega$, but it is only lower semicontinous with respect to the weak topology. Throughout this paper we will use the shorthand notation: \be\Omega^E=\{\gamma\in\Omega\,|\,J(\gamma)\leq E\}.\ee

The \emph{extended} endpoint map  $\varphi:\Omega\to M\times \R$ is defined by (notice that it depends on the choice of the subriemannian structure on $\Delta$):
\be \gamma\mapsto (F(\gamma), J(\gamma)).\ee

\subsection{Normal extremals} Once we have fixed a riemannian metric $g$ on $\Delta$, we can define the subriemannian hamiltonian $H:T^*M\to\R$ as follows.  We require that for every $x\in M$, the restriction of $H$ to the fiber $T^*_xM$ coincides with the nonnegative quadratic form
\be
	\lambda\mapsto\frac12\max\left\{\frac{\langle\lambda,v\rangle^2}{g_q(v,v)}\,\bigg|\, v\in\Delta_x\setminus\{0\}\right\}.
\ee
We call \emph{normal extremal} any integral curve of the vector field\footnote{We recall that $\overrightarrow{H}$ is defined by the equation $\iota_{\overrightarrow{H}}\omega=-dH$} $\overrightarrow{H}$, that is any curve $\lambda:[0,1]\to T^*M$ such that $\dot{\lambda}(t)=\overrightarrow{H}(\lambda(t))$. The Pontryagin maximum principle \cite{AgrachevSarychev} states that a necessary condition for an horizontal curve to be locally minimizing is to be either the projection of a normal or an abnormal extremal (the two possibilities are not mutually exclusive in principle); accordingly, a singular curve $\gamma$ which is not the projection of a normal extremal will be said \emph{strictly abnormal}.
\subsection{The Goh condition} Let us consider a singular curve $\gamma:I\to M$, and let $\lambda:I\to T^*M$ an abnormal lift for $\gamma$. We say that $\gamma$ is a \emph{Goh} singular curve if $\lambda(t)\in \left(\Delta_{\gamma(t)}^2\right)^\perp$. A detailed discussion of the Goh condition can be found in \cite[Chapter 11]{AgrachevBarilariBoscain}; here we just briefly recall the salient facts. Since $\gamma$ is a critical point of the endpoint map, there is a well-defined map (the \emph{Hessian}) such that: \be \textrm{Hess}_\gamma F:\ker d_\gamma F\to \textrm{coker}\, d_\gamma F=T_{F(\gamma)}M/\textrm{im}\,d_\gamma F.\ee If we precompose the Hessian with $\lambda=\lambda(1)\in T^*_{F(\gamma)}M$ what we get is a real-valued quadratic form defined on $\ker d_\gamma F$. A necessary condition for this map to have finite negative inertia index, i.e.  $\textrm{ind}^-\lambda\textrm{Hess}_\gamma F<+\infty$, is that $\gamma$ has to be a Goh singular curve. In this sense we may think of the Goh condition as a necessary second-order optimality condition for singular curves.
\subsection{The global ``chart'' and the minimal control}\label{sec:global}We discuss in this section a useful construction introduced in \cite{LerarioMondino} in order to switch from curves to controls.

Assume $M$ is a compact manifold. Given a subriemannian structure on $\Delta\subset TM$, there exists a family of vector fields $X_1, \ldots, X_l$ with $l\geq d$ such that:
\be \Delta_x=\textrm{span}\{X_1(x), \ldots, X_l(x)\}, \qquad \forall x\in M.\ee
Moreover the previous family of vector fields can be chosen such that for all $x\in M$ and $u\in \Delta_x$ we have \cite[Corollary 3.26]{AgrachevBarilariBoscain}:
\be\label{eq:uminglob}
|u|^2=\inf \left\{u_1^2+\cdots + u_l^2\,\bigg|\, u=\sum_{i=1}^l u_i X_i(x)\right\},
\ee
where $|\cdot|$ denotes the modulus w.r.t. the fixed subriemannian structure.
Denoting by $\mathcal{U}=L^2(I, \R^l)$,  we define the map $A: {\mathcal U}\to \Omega$ by:
\be\label{eq:defA}
 A(u)=\textrm{the curve solving the Cauchy problem $\dot\gamma =\sum_{i=1}^lu_i(t)X_i(\gamma(t))$ and $\gamma(0)=x$}.
 \ee
(We can use the compactness of $M$ to guarantee that the solution to the Cauchy problem is defined for all $t\in [0,1]$, otherwise we need to define $\mathcal{U}$ as the open set of controls for which the solution $A(u)$ is defined up to time $t=1$).

We will consider this construction fixed once and for all,  and call it the ``global chart''. Abusing of notation, the endpoint map for this global chart will still be denoted by $F:\mathcal{U}\to M$:
\be\label{eq:globalendpoint}
 F:\mathcal{U}\to M, \quad F(u)=F(A(u)).
\ee

The map $A$ is continuous (both for the strong and the weak topologies on $L^{2}(I, \R^l)$) and has a right inverse $\mu:\Omega \to \mathcal{U}$ defined by:
\be \mu(\gamma)=u^*(\gamma)\ee
where $u^*(\gamma)$ is the control realizing the minimum of $\|\cdot\|^2$ on $A^{-1}(\gamma)$ (notice that this in particular implies $J(A(u))\leq J(u)$). This control is called the \emph{minimal control}  \cite[Remark 3.9]{AgrachevBarilariBoscain}. The minimal control exists and is unique by \cite[Lemma 2]{LerarioMondino}; it depends continuously on the curve $\gamma$ (for the strong topologies) by \cite[Proposition 4]{LerarioMondino}.
Moreover \cite[Lemma 3]{LerarioMondino} gurarantees that:
\be \label{eq:Ju*}
 J(\gamma)=\frac{1}{2}\|u^*(\gamma)\|^2.\\
\ee

In the sequel, given a point $y\in M$ we will fix a set of coordinates on a neighborhood $U\simeq \R^m$ of the point $y$, denote by $\mathcal{U}=F^{-1}(U)$ and simply write
\be \varphi:\mathcal{U}\to U\times \R\subset \R^{m+1}\ee
for the extended endpoint map already in coordinates.

A singular control $u$ is a critical point of $F:\mathcal{U}\to M$; its corank is the corank of $d_uF$ (notice that the corank of $u$ can also be defined as the corank of $F|_{\{J=J(u)\}}$).

\subsection{How to build a subriemannian manifold}
Following the notation of \cite{AgrachevBarilariBoscain}, we will assume that the distribution $\Delta\subset TM$ is defined as the image of a bundle map with constant rank:
\be f:M\times \mathbb{R}^l\to TM, \quad \Delta_x=f(U_x)\ee
In this way $\Delta$ can be endowed with a subriemannian metric by simply taking \eqref{eq:uminglob} as a definition.
By virtue of \cite[Corollary 3.26]{AgrachevBarilariBoscain}, this construction is indeed equivalent to the standard one. Denoting by $\{e_1, \ldots, e_l\}$ the standard basis of $\R^{l}$, this approach also has the advantage that the vector fields  generating the distribution are naturally defined as:
\be X_i(x)=f(x,e_i), \quad x\in M.\ee

\subsection{Generic properties and soft curves}\label{subsect:genericprop}
In addition to the totally nonholonomic condition on $\Delta$ (also called the H\"ormander condition \cite{AgrachevBarilariBoscain}), in this paper we will consider subriemannian structures whose singular curves satisfy the following properties.

\begin{defi}[Soft singular curves]\label{def:soft}
We will say that a singular curve is \emph{soft} if: \emph{(a)} it has corank one, \emph{(b)} it is not Goh and \emph{(c)} it is strictly abnormal. We will use the same terminology for singular controls.
\end{defi}

For \emph{generic} subriemannian structures all singular curves are soft, as it is clarified by the following result from \cite{CJT}. We denote by $\mathcal{D}_d$ the set of rank $d$ distributions on $M$ endowed with the Whitney $C^\infty$ topology and by $\mathcal{G}_d$ the set of of couples $(\Delta, g)$ where $\Delta$ is a distribution on $M$ and $g$ is a Riemannian metric on $\Delta$, endowed with the Whitney $C^{\infty}$ topology. We will say that a distribution $\Delta\subset TM$ satisfy a property from (a), (b), (c) if all its singular curves satisfy this property.
		\begin{thm}\label{thm:CJT}If $d\geq 2$ there exists an open dense set $O_{a,d}\subset \mathcal{D}_m$ where condition \emph{(a)} is satisfied \cite[Theorem 2.4]{CJT}; if $d\geq 3$ there exists an open dense set $O_{b,d}\subset \mathcal{D}_d$ where also condition \emph{(b)} is satisfied \cite[Corollary 2.5]{CJT}. Moreover, if $d\geq 2$ there exists an open dense set $O_{c,d}\subset \mathcal{G}_m$ where condition \emph{(c)} is satisfied \cite[Proposition 2.7]{CJT}. In particular for a generic subriemannian structure of rank $d\geq 3$ all singular curves are soft.
		\end{thm}
Using the approach of viewing the subriemannian structure as the image of a bundle map, the above conditions are still generically satisfied. Specifically, when working in the global chart, we consider the control system $\{X_1, \ldots, X_l\}$ whose \emph{trajectories} are solutions to:
\be \dot x=\sum_{i=1}^lu_i X_i, \quad x(0)=x.\ee
The fact that generically in the Whitney topology for $l$-ples of vector fields on $M$ all singular controls are soft is granted by \cite[Theorem 2.6, Corollary 2.7]{CJT2}.	
Moreover it is not difficult to verify that if a control is not soft then the same is true for the associated trajectory. In fact let $G:\mathcal U\to M$ be defined as $G=F\circ A$, and assume that $u$ is a critical point for $G$. Then the following chain of implications holds \be 0=\lambda d_u G\Leftrightarrow \lambda d_{A(u)}F\circ d_uA=0\Leftrightarrow \lambda d_{A(u)}F=0,\ee where we have used in the last implication that $A$ is surjective. Then any abnormal lift of $u$ is an abnormal lift for the (singular) curve $\gamma$, hence it must be unique, and it must annihilate $\Delta_{\gamma(t)}^2$ for every $t\in [0,1]$ whence (a) and (b) follows also for controls. To prove (c) observe that if the singular control $u$ admits a normal extremal lift, then we must have $\lambda d_uG=d_u J=u$, and in particular the curve $A(u)$ has to be a local minimizer of the length, parametrized with constant velocity. The Pontryagin maximum principle says that in this case $A(u)$ is either the projection of a normal or an abnormal extremal, but the first possibility is excluded by point (c) for curves. Hence $u$ cannot admit a normal extremal lift (it would also be a normal extremal lift of $A(u)$), and it is therefore strictly abnormal. (As a corollary, if all singular curves are soft, then the same is true for all singular controls.)
	
The set $\Omega$ can be decomposed as follows: \be\label{eq:decomp}\Omega=\mathcal R\cup \mathcal C\cup \mathcal A,\ee where $\mathcal R$ is the set of regular points for $\varphi$, $\mathcal{C}$ consists of \emph{strictly normal} curves, i.e.  regular points $\gamma$ of $F$ for which there exists $(\lambda_0, \lambda)$ with $\lambda_0\neq 0$ such that $\lambda d_\gamma F=\lambda_0d_\gamma J$, and $\mathcal A$ are the \emph{abnormal} curves, for which there exists $(0,\lambda)$ such that $\lambda d_\gamma F=0$. If properties (a) and (c) are verified, these three sets are indeed disjoint.

All the ``technical'' results in the sequel will be proved for the global chart, but we will go back to the general setting of horizontal curves for the main theorems.
Abusing of notations, we will still denote by $\mathcal{C}$ the set of strictly normal \emph{controls} and by $\mathcal{A}$ the set of abnormal \emph{controls}.

	\section{soft abnormal controls: the cross section}\label{sec:cross}
Let $u_0$ be a \emph{soft} abnormal control with Energy $J(u_0)\leq E$. Then the corank of $\varphi$ at a $u_0$ is one and there exist $e_1(u_0), \ldots, e_m(u_0)$ such that:
\be \textrm{im}d_{u_0}\varphi =\textrm{span}\{d_{u_0}\varphi \,e_1(u_0), \ldots, d_{u_0}\varphi \,e_m(u_0)\}.\ee
  Consider now the finite co-dimensional set: \be\label{eq:V} \mathcal P=\left\{v\in L^2(I,\R^l)\,\bigg |\,\|v\|_{L^2}\leq 1\,\:\text{and}\: \int_0^1v(t)dt=0\right\}.\ee For $v\in\mathcal P$,  $\bt\in[0,1]$ and $s\in \R$ small enough, we set: \be\label{eq:vs} v_s(t)=\left\{\begin{array}{cc}\frac{1}{|s|^{1/4}}v\left(\frac{t-\bt}{|s|^{3/4}}\right)&\overline t\leq t\leq \overline t+|s|^{3/4}\\
	&\\
	0&\textrm{otherwise}\end{array}\right.\ee
	An easy computation shows that for any $v\in\mathcal P$ we have:
	\be\label{eq:norm}\|v_s\|_{L^2}=|s|^{1/8}\|v\|_{L^2}\xrightarrow{s\to 0} 0\ee
	
	For any $u\in\Lp{2}{I,\R^l}$ we define the non autonomous horizontal vector field $f_{u}=\sum_{i=1}^{l}u^iX_i$; if $0\leq t_1\leq t_2\leq 1$, its flow is given by the diffeomorphism (we highlight from here on the explicit dependence on $u$): \be\label{put}\Put{t_1}{t_2,u}=\eexp\int_{t_1}^{t_2}f_{u(t)}dt.\ee

Since $u_0$ is not a Goh singular control, there exist $\bt\in[0,1/2]$\footnote{Even small pieces of singular curves cannot be Goh curves} and $a,b\in\R^m$ such that \be \label{eq:nonGoh}\langle(\Put{\bt}{1,u_0})^*\lambda,[f_{a},f_{b}]\rangle\neq 0.\ee This in turn yields \cite[Lemma 11.21]{AgrachevBarilariBoscain} that the map $Q^{u_0}:\mathcal P\to\R$, defined by \be\label{eq:Q}Q^{u_0}(v(\cdot))=\int_0^1\langle(\Put{\bt}{1,u_0})^*\lambda,[f_{w(\theta)},f_{v(\theta)}]\rangle d\theta,\ee with $w(\theta)=\int_0^\theta v(\zeta)d\zeta$, has infinite positive and negative index. It is then possible to choose $v^+$ and $v^-$ in $\mathcal P$ so that $\textrm{sign}(Q^{u_0}(v^{\pm}))=\pm 1$.

With this notation, we define:
		\be\label{eq:sigma}\alpha_{u_0}(x,y)=v_{|x|}^{\textrm{sgn}(x)}+y_1e_1(u_0)+\cdots +y_m e_m(u_0),\ee where in the definition of $v_{|x|}^{\textrm{sgn}(x)}$ (compare with \eqref{eq:vs}) we choose $\bt$ as a time for which \eqref{eq:nonGoh} holds.

The goal of this section is to prove the following key proposition.

\begin{prop}\label{prop:proof}For every soft $u_0\in \mathcal{A}^E$ consider the function:
\be\label{equation} G_{u_0}(u,x,y)=\varphi(u+\alpha_{u_0}(x,y))-\varphi(u)\ee
where $\alpha_{u_0}$ is defined as above.
There exist weak neighborhoods $\mathcal{V}\subset \mathcal{W}$ of $u_0$, positive constants $r_{1}, r_{2}, r_3>0$ and a function:
\be g:B(0, r_1)\times \mathcal{W}\to B(0, r_3)\ee
which is continuous for the strong topology and such that for every $(w, u)\in B(0, r_1)\times \mathcal{W}$
\be G_{u_0}(u,g(w,u))=w.\ee
Moreover for every $u\in \mathcal{V}\cap \{J\leq E\}$ we have $B(u, r_2)\subset \mathcal{W}$.
\end{prop}
We postpone the proof to Section \ref{sec:proof}, because it will require some preliminary results.
\subsection{A change of coordinates}

Let us start by writing:
\be G_{u_0}(u, x, y)=(\varphi_0(u, x,y), \varphi_1(u,x,y), \ldots, \varphi_{m}(u, x, y))-\varphi(u)\ee
and for $i=1, \ldots, m$ let us denote by $\psi_i^u:\R^{m+1}\to \R$ the function:
\be \psi_i^u(x,y)=\varphi_i(u, x, y)-\varphi_i(u).\ee
Notice that the map $G_{u_0}$ is weak continuous: the first components are continuous because the Endpoint map itself is continuous; for the last component we have:
\begin{align} \psi_m^u(x,y)&=\frac{1}{2}\left(\|u+\alpha_{u_0}(x,y)\|^2-\|u\|^2\right)\\
&=\frac{1}{2}\left(2\langle u, \alpha_{u_0}(x,y)\rangle +\|\alpha_{u_0}(x,y)\|^2\right)\end{align}
which is weak-continuous as a function of $(u,x,y)$.
We will consider local coordinates on a neighborhood $W\simeq \R^{m+1}$ of zero induced by the splitting:
		\be \label{eq:coc}\R^{m+1}\simeq\textrm{coker} d_{u_0}\varphi\oplus  \textrm{im} d_{u_0}\varphi.\ee
\begin{lemma}\label{lemma:coord}For every soft $u_0\in \mathcal{A}^E$ there exist a weak neighborhood $\mathcal{W}_1$ of $u_0$, a ball $B_1\subset \R^m$ centered at zero and an interval $I_1\subset \R$ centered at zero such that for every $u\in \mathcal{W}_1$ the map:
\begin{align} \phi^u=(x, \psi_1^u, \ldots, \psi_m^u)&:I_1\times B_1\to \R^{m+1}\\
&(x,y)\mapsto(x, \psi_1^u(x,y), \ldots, \psi_m^u(x,y))
\end{align}
is a coordinate chart (i.e. a homeomorphism onto its image).

Moreover there exist positive numbers $r_1', r_1''>0$ and a weak neighborhood $\mathcal{V}_1\subset \mathcal{W}_1$ such that $\phi^u(I_1\times B_1)\supset I_1\times B(0, r_1')$  and $B(u, r_1'')\subset \mathcal{W}_1$ for every $u\in \mathcal{V}_1\cap \{J\leq E\}$.
\end{lemma}

\begin{proof} Denoting by $(G_0, G_1, \ldots, G_m)$ the components of $G_{u_0}$, notice that the partial derivatives $\partial_{y_j}G_i|_{(u, x, y)}$ exist for $i, j=1, \ldots, m$ and are continuous functions for the \emph{weak} topology on the $u$-variable. In fact there exists a matrix $A$ (because we have performed the change of coordinates \eqref{eq:coc} in the target space) such that:
\be\label{G} \left(\frac{\partial G_i}{\partial y_j}\bigg|_{(u, x, y)}\right)_{i, j=1, \ldots, m}=A\cdot \left(\begin{array}{ccccc}\frac{1}{2}\langle u+\alpha_{u_0}(x,y), e_1(u_0)\rangle &  & \cdots &  & \frac{1}{2}\langle u+\alpha_{u_0}(x,y), e_m(u_0)\rangle \\ d_{u+\alpha_{u_0}(x,y)}F\,e_1(u_0) &  & \cdots &  & d_{u+\alpha_{u_0}(x,y)}F\, e_m(u_0) \\ &  &  &  &  \\\vdots &  & \vdots &  & \vdots \\ &  &  &  &  \\d_{u+\alpha_{u_0}(x,y)}F\, e_1(u_0) &  & \cdots &  & d_{u+\alpha_{u_0}(x,y)}F\, e_m(u_0)\end{array}\right).\ee
The elements of the first row in the above matrix are fixed linear functional evaluated on $u+\alpha_{u_0}(x,y)$ and are continuous in $(u,x,y)$ even when considering the weak topology for the $u$-variable; for all the other elements, the weak-continuity follows from the fact that the map $u\mapsto d_u F$ is weak-strong continuous.
We denote by:
\be D(u,x, y)=\left(\frac{\partial G_i}{\partial y_j}\bigg|_{(u, x, y)}\right)_{i, j=1, \ldots, m}.\ee
By assumption the matrix $D(u_0, 0,0)$ is of rank $m$ and there exists a ball $W$ centered at zero in $\R^m$ such that $\psi^{u_0}(0, \cdot):W\to \R^m$ is a diffeomorphism onto its image.

Consider the function $H:L^2\times \R\times \overline W\times S^{m-1}\to \R$ defined by:
\be H(u, x, y, v)=\max_{i=1,\ldots, m}\left\{|\langle \nabla_y\varphi_i|_{(u, x, y)}, v\rangle|\right\}\ee
where $ \nabla_y\varphi_i|_{(u, x, y)}$ denotes the $i$-th row of $D(u, x, y)$. Note that this function is continuous with respect to the weak topology in the $u$-variable.

We claim that there exists $a>0$ such that $H(u_0, 0, 0, v)>a$ for every $v\in S^{m-1}$. Assume that this is false. Then, for every $n\in \mathbb{N}$ there exists $v_n\in S^{m-1}$ such that $|\langle \nabla_y\varphi_{i}|_{(u_0, 0, 0)}, v_n\rangle|\leq 1/n$ for every $i=1, \ldots, m$. By compactness of $S^{m-1}$ this gives the existence of a nonzero vector $\lim_k v_{n_k}=\overline v\in S^{m-1}$ such that $|\langle \nabla \varphi_{i}|_{(u_0, 0, 0)}, \overline v\rangle|=0$, which means that $\overline v$ is orthogonal to all the rows of $D(u_0, 0, 0)$ which is impossible since this matrix is invertible.

This implies that there exists a weak-open neighborhood $U^{\textrm{weak}}=\mathcal{U}_1^{\textrm{weak}}\times U_1\times W_1$ of $(u_0, 0, 0)$ (here we take $W_1$ to be a ball for simplicity)
such that for every $v\in S^{m-1}$ there exists $i=i(v)\in\{1, \ldots, m\}$ such that  $|\langle \nabla_y\varphi_i|_{(u, x, y)}, v\rangle|>a$ for all $(u, x,y)\in U^{\textrm{weak}}$. In fact, by weak-continuity of $H$, for every $v\in S^{m-1}$ there exists a weak-open neighborhood $W(v)^{\textrm{weak}}=\mathcal{U}_1(v)^{\textrm{weak}}\times U_1(v)\times W_1(v)\times U_{S^{m-1}}(v)$ with the property that there exists  $i\in \{1, \ldots, m\}$ such that for every  $(u,x,y,w)\in W(v)^{\textrm{weak}}$  we have
$|\langle \nabla_y\varphi_i|_{(u, x, y)}, w\rangle|>a$. By compactness of $S^{m-1}$ there exist $v_1, \ldots, v_N$ such that $\{U_{S^{m-1}}(v_j)\}_{j=1,\ldots, N}$ is an open cover of $S^{m-1}$. The open set $U^{\textrm{weak}}$ is defined as:
\be U^{\textrm{weak}}=\bigcap_{j=1}^N \left(\mathcal{U}_1(v_j)^{\textrm{weak}}\times U_1(v_j)\times W_1(v_j)\right).\ee
We use this to prove that for every $(u,x)\in \mathcal{U}_{1}^{\textrm{weak}}\times U_1$ the following map is injective:
\be \psi^u(x, \cdot)=(\psi_1(u, x, \cdot), \ldots, \psi_m(u, x, \cdot)):W_1\to\R^m.\ee
To this end consider:
\begin{align}
\|\psi^{u}(x, y_1)-\psi^{u}(x, y_2)\|_1&=\sum_{i=1}^m\left|\varphi_i(u, x, y_1)-\varphi_i(u, x, y_2)\right|\\
&=\sum_{i=1}^m\left|\int_{0}^1\partial_t f_i(u,x,t)dt\right|=(*)
\end{align}
where $f_i(u,x,t)=\varphi_i(u, x, y_2+t(y_1-y_2))$. Consequently:
\be \partial_tf_i(u, x, t)=\langle \nabla_y\varphi_i|_{(u, x, y_2+t(y_1-y_2))}, y_1-y_2\rangle=\|y_1-y_2\|\langle \nabla_y\varphi_i|_{(u, x, y(t))}, v\rangle\ee
where we have set $y(t)=y_2+t(y_1-y_2)\in W_1$ and $v=\frac{y_1-y_2}{\|y_1-y_2\|}\in S^{m-1}.$ Thus we can write:
\begin{align}(*)&=\sum_{i=1}^m\left|\int_{0}^1\|y_1-y_2\|\langle \nabla_y\varphi_i|_{(u, x, y(t))}, v\rangle dt\right|\\
&\geq \left|\int_{0}^1\|y_1-y_2\|\langle \nabla_y\varphi_{i(v)}|_{(u, x, y(t))}, v\rangle dt\right|\\
&\geq  \|y_1-y_2\| a
\end{align}
which proves the injectivity of $\psi^u(x, \cdot).$

As a consequence for every $u\in \mathcal{U}_1^{\textrm{weak}}$ the map:
\be \phi^u=(x, \psi^u):U_1\times W_1\to \R^{m+1}\ee
is continuous and injective; by the Invariance of Domain Theorem this map is a homeomorphism onto its image.

We prove now that (up to restricting the open set $U^{\textrm{weak}}$) the image of $(x, \psi^u)$ contains a ball of uniform radius.

First, notice that the above chain of inequalities implies that for every $u\in \mathcal{U}_1^{\textrm{weak}}$ and $x\in U_1$ we have:
\be \|\psi^{u}(x, y_1)-\psi^{u}(x, y_2)\|\geq c a \|y_1-y_2\|\ee
where a constant $c>0$ appears, but it only depends on $m$ (because all norms on a finite dimensional space are quivalent).
We are now in the position of using \cite[Lemma 5]{Clarke}, which guarantees that if $rB\subset W_1$ (here $rB$ is a shorthand notation for $B(0, r)$), then:
\be \psi^u(x, \cdot)(rB)\supset \psi^{u}(x, 0)+rcaB.\ee
The map $G_{u_0}$ is weak continuous and $\psi^u(0,0)=0$. Hence for every $r>0$ there exists a weak open set $\mathcal{O}_1^{\textrm{weak}}(r)\times O_1(r)$ such that for every $(u,x)\in \mathcal{O}_1^{\textrm{weak}}(r)\times O_1(r)$ we have $\|\psi^u(x,0)\|<\frac{1}{3}acr.$
In particular for every $u\in\mathcal{U}^{\textrm{weak}}_1\cap \mathcal{O}_1^{\textrm{weak}}(r) $ and for every $x\in U_1\cap O_1(r)$, if $rB\subset W_1$ then:
\be \psi^u(x, \cdot)(rB)\supset \psi^{u}(x, 0)+racB\supset \frac{2}{3}racB.\ee
Let now $\tilde r>0$ be such that\footnote{Here we chose $\tilde r $ such that $2\tilde r\subset W_1$ in order to guarantee that:
\be\label{eq:clos} \textrm{clos}\left((\phi^u)^{-1}(U_3\times 2\tilde rac/3B)\right)\subset \textrm{clos}\left(U_3\times \tilde r B\right).\ee
We will need this property in the proof of Corollary \ref{cor:impl1}.} $2\tilde rB\subset W_1$ and denote by $\mathcal{U}_3^{\textrm{weak}}= \mathcal{U}_1^{\textrm{weak}} \cap \mathcal{O}_1^{\textrm{weak}}(\tilde r) $ and by $U_3= U_1\cap O_1(\tilde r).$
Then we have just proved that, for every $u\in\mathcal{U}_3^{\textrm{weak}}$, the map:
\be \phi^u:U_3\times \tilde rB\to \R^{m+1}\ee
is a homeomorphism onto its image, and this image contains $U_3\times \frac{2}{3}\tilde racB.$

Consider now the two functions:
\be \alpha_1(u)=\min_{(x, y, v)\in \overline U_3\times\tilde r\overline B\times S^{m-1}}H(u, x, y, v)\quad \textrm{and}\quad \alpha_2(u)=\max_{x\in \overline {U}_3}\|\psi^u(x, 0)\|.\ee
These functions are well defined (the max and the min are taken over compacts) and are continuous for the weak topology\footnote{This follows from this elementary fact. Let $F:P\times K\to \R$ be a continuous function, where $P$ and $K$ are topological spaces and $K$ is compact. Define $f(p)=\max_{k\in K}F(p,k)$. Then $f$ is continuous (and the analogue statement  with $\max$ replaced with $\min$ is also true). The proof is easy and left to the reader.}. Moreover:
\be \mathcal{U}_1^{\textrm{weak}}\supset\{\alpha_1>a\}\quad \textrm{and}\quad \mathcal{O}_{1}^{\textrm{weak}}(\tilde r)\supset\{\alpha_2<ac\tilde r/3\}.\ee
Define finally, for $\epsilon>0$ small enough, the open sets:
\be \mathcal{W}_1=\{\alpha_1>a, \alpha_2<ac\tilde r/3\},\quad \mathcal{V}_1=\{\alpha_1>a+\epsilon, \alpha_2<ac\tilde r/3-\epsilon\}, \quad I_1=U_3\quad B_1=\tilde rB. \ee
The two open sets $\mathcal{W}_1, \mathcal{V}_1$ are weakly open because $\alpha_1, \alpha_2$ are weak continuous; $\epsilon>0$ is taken small enough in order to guarantee that $u_0\in \mathcal{W}_1, \mathcal{V}_1$ (such an $\epsilon$ exists because $\alpha_1(u_0)>a$ and $\alpha_2(u_0)<ac\tilde r/3$).

Then $r_1'=\frac{2}{3}\tilde rac$ satisfies the requirements from the statement. For the existence of $r_1''>0$ we argue as follows. We consider the weak closed set $C=(\mathcal{W}_1)^c$ and the weak compact:
\be K=  \{\alpha_1\geq a+\epsilon, \alpha_2\leq a c \tilde r/3-\epsilon\}\cap\{J\leq E\}.\ee
These two sets are disjoint and by Lemma \ref{lemma:wdist} below there exists $r_1''>0$ such that each ball of radius $r_1''$ centered on some $u\in K$ is entirely disjoint from $C$, which means it is contained in $\mathcal{W}_1$. This concludes the proof.
\end{proof}
\begin{lemma}
	\label{lemma:wdist}
	Let $X$ be a normed space, $C\subset X$ be weakly closed and $K\subset X$ weakly compact, and assume that $C\cap K= \emptyset$. Then there exists $\nu>0$ such that
	\be
	\label{eq:wdist}
	\textrm{dist}(C,K)=\inf\left\{\|u-v\|_X\,|\,u\in C,\, v\in K\right\}>\nu.
	\ee
\end{lemma}

\begin{proof}
	Let us suppose, on the contrary, that $\textrm{dist}(C,K)=0$. Then we can find sequences $\{u_n\}_{n\in\N}\subset C$ and $\{v_n\}_{n\in\N}\subset K$ such that, for every $n\in \N$, we have
	\be
	\|u_n-v_n\|_X<\frac{1}{n}.
	\ee
	Since $K$ is weakly compact, by the Eberlein-Smulian theorem it is also sequentially weakly compact, and there exists $v\in K$ such that $v_n\rightharpoonup v$. We claim that actually $v$ is a weak limit also for the sequence $\{u_n\}_{n\in\N}$, and we will have the absurd since then $v$ would be forced to be also an element of $C$. But this easily follows from the fact that, if $\Lambda$ is any norm-one linear functional on $X$, the following line holds
	\be
	\Lambda(u_n-v)=\Lambda(u_n-v_n)+\Lambda(v_n-v)\leq \|\Lambda\|_{X^*}\|u_n-v_n\|_X+\Lambda(v_n-v)\leq \frac{1}{n}+\Lambda(v_n-v)\to 0.
	\ee
\end{proof}
\begin{cor}\label{cor:impl1}Keeping the notation of Lemma \ref{lemma:coord}, for every soft $u_0\in \mathcal{A}^E$ the function:
\be g_1:I_1\times B(0, r'_1)\times \mathcal{W}_1\to \overline I_1\times \overline B_1\ee
giving for every $(x,y, u)\in I_1\times B(0, r'_1)\times \mathcal{W}_1$ the unique solution to:
\be\label{eq:sol}\phi^u(g_1(x,y, u))=(x,y),\ee
is well defined and continuous (for the strong topologies).

\end{cor}
\begin{proof}For every $u\in \mathcal{W}_1$ the inverse of $\phi^u$ is defined on $\overline I_1\times \overline B_1$ and continuous, hence $g_1(x,y, u)$ is well defined.
To prove that it is continuous for the strong topology, consider a sequence $\{(x_n, y_n, u_n)\}_{n}\subset I_1\times B(0, r'_1)\times \mathcal{W}_1$ converging to $(\overline x, \overline y, \overline u)\in I_1\times B(0, r'_1)\times \mathcal{W}_1 $. Then:
\be \phi^{u_n}(g_1(x_n, y_n, u_n))=(x_n, y_n).\ee
Denote by $g^n=g_1(x_n, y_n, u_n)$; since $\overline I_1\times \overline B_1$ is compact, let $g^{n_k}\to \overline g\in \overline I_1\times \overline B_1$. Then, by continuity of $\phi$ we have:
\be (\overline x, \overline y)=\lim_{k\to\infty}\phi^{u_{n_k}}(g^{n_k})=\lim_{k\to \infty}\phi^{\overline u}(\overline g)\ee
which proves $\overline g$ is a solution to \eqref{eq:sol}; this solution is unique (because of \eqref{eq:clos}), hence:
\be \overline g=g_1(\overline x, \overline y, \overline u).\ee
This proves that the bounded sequence $\{g^n\}_n$ has only one accumulation point $\overline g$, hence the all sequence converges itself to $\overline g$.
\end{proof}
\subsection{Lipschitz inverses}
The previous section provided a convenient change of coordinates to express the last $m$ components of $G_{u_0}(u,x,y)$; indeed these were linearized after we changed the differentiable structure, and the equation $G_{u_0}(u,x,y)=w$ can be reinterpreted as
\be
	\label{eqnewcoord}
	G_{u_0}(u,x,\psi)=(\varphi_0(u,x,\psi)-\varphi_0(u,0,0),\psi^u_1,\dotso,\psi^u_m)=w.
\ee
It is evident how to choose the $\psi$ coordinates in order to solve \eqref{eqnewcoord}: this means that the problem is reduced at this point in finding a continuous solution to the single real equation:
\be
	x\mapsto \varphi_0(u,x,\psi^u)-\varphi_0(u,0,0).
\ee
Observe at first that if we want to study the $x$-derivative of $G_{u_0}(u,x,\psi)$, we may reduce ourselves to the case $\psi^u=0$. Indeed, if we switch back for a moment to the $(x,y)$-coordinates, we see that
\begin{align}
	\varphi_0(u,x,y)&=\overline{\lambda}\varphi(u+\alpha_{u_0}(x,y))\\&=\lambda F(u+\alpha_{u_0}(x,y))\\&=\lambda F(u+v_{|x|}^{\textrm{sgn}(x)}+y_1e_1(u_0)+\dotso+y_me_m(u_0)),
\end{align}
where the equality in the second line follows from the fact that $\overline{\lambda}=(\lambda,0)$, by the corank one assumption on singular curves. Then, the identity
\be
	\frac{\partial \varphi_0}{\partial x}\bigg|_{(u,x,y)}=\frac{\partial \varphi_0}{\partial x}\bigg|_{(\widetilde{u},x,0)},\quad\textrm{with }\widetilde{u}=u+y_1e_1(u_0)+\dotso+y_me_m(u_0),
\ee
shows that the $x$-derivative of $G_{u_0}$ at the point $(u,x,y)$ coincides with the $x$-derivative of $G_{u_0}$ evaluated at the point $(\widetilde{u},x,0)$. Finally, notice that equation \eqref{G} implicitly defines $\psi^u(x,0)=0$.

Let us fix $v\in\mathcal P$ and $s_0\in\R$; our objective is to evaluate the limit
\be \label{limit}
	\lim_{s\to 0}\frac{F(u+v_{s+s_0})-F(u+v_{s_0})}{s}.
\ee
By the results in \cite[Section 2.6]{AgrachevSarychev}, the previous limit exists if and only if \be\lim_{s\to 0}\frac{G(u+v_{s+s_0})-G(u+v_{s_0})}{s}\ee exists, with (observe that we give the inline definition of $\gut{v}{t,u}$)
\be
G(v_q)=\eexp\int_{\bt}^{\bt+|q|^{3/4}}(\Put{t}{0,u})_*f_{v_q(t)}dt=\eexp\int_{\bt}^{\bt+|q|^{3/4}}\gut{v_q(t)}{t,u}dt;
\ee
moreover, the following identity holds:
\be
\lim_{s\to 0}\frac{F(u+v_{s+s_0})-F(u+v_{s_0})}{s}=\left(\Put{0}{1,u}\right)_*\left(\lim_{s\to 0}\frac{G(v_{s+s_0})-G(v_{s_0})}{s}\right).
\ee
The Volterra series \cite[Section 2.4]{AgrachevSarychev} provides the expansion:
\be
G(v_q)=Id+\int_{\bt}^{\bt+|q|^{3/4}}\gut{v_q(t)}{t,u}dt+\iint\limits_{\bt\leq \tau\leq t\leq \bt+|q|^{3/4}}\gut{v_q(\tau)}{\tau,u}\circ\gut{v_q(t)}{t,u}d\tau dt+|q|^{3/2}O(1),
\ee
where the last term is a consequence of the fact that $\|v\|_{L^2}\leq 1$, and the change of variables:
\be
	\theta=\frac{t-\bt}{|q|^{3/4}},
\ee
(it weakly depends on $u$, since so does the flow $u\mapsto \Put{t_1}{t_2,u}$). The behavior of \eqref{limit} will be thus determined by a careful analysis, as $s\to 0$, of the expression
\begin{align} \label{expansiong}
	\frac{G(v_{s+s_0})-G(v_{s_0})}{s}&=\frac{1}{s}\bigg(\int_{\bt}^{\bt+|s+s_0|^{3/4}}\gut{v_{s+s_0}(t)}{t,u}dt-\int_{\bt}^{\bt+|s_0|^{3/4}}\gut{v_{s_0}(t)}{t,u}dt\\&+\iint\limits_{\bt\leq \tau\leq t\leq \bt+|s+s_0|^{3/4}}\gut{v_{s+s_0}(\tau)}{\tau,u}\circ\gut{v_{s+s_0}(t)}{t,u}d\tau dt-\iint\limits_{\bt\leq \tau\leq t\leq \bt+|s_0|^{3/4}}\gut{v_{s_0}(\tau)}{\tau,u}\circ\gut{v_{s_0}(t)}{t,u}d\tau dt\\&+(|s+s_0|^{3/2}-|s_0|^{3/2})O(1)\bigg)\\&=\frac{1}{s}\left(I'(s)+I''(s)+(|s+s_0|^{3/2}-|s_0|^{3/2})O(1)\right)
\end{align}

\subsubsection{Estimate on the remainders} In subsequent calculations, we will frequently use the fact that, for any fixed $u,v\in \Lp{2}{I,\R^l}$, the map $t\mapsto \gut{v}{t,u}$ is Lipschitzian. The following lemma gives a quantitative version of this statement.
\begin{lemma}
	\label{lemma:lipschitz}
	Let $u,v\in\Lp{2}{I,\R^l}$, and let $g_1(s,s_0)$ and $g_2(s,s_0)$ be any two real numbers depending on $s$ and $s_0$. Then we have the estimate
	\be
	\int_0^1\left|\gut{v(\theta)}{\bt+g_1(s,s_0)\theta}-\gut{v(\theta)}{\bt+g_2(s,s_0)\theta}\right|d\theta\leq \sqrt{l}a(u)|g_1(s,s_0)-g_2(s,s_0)|\|v\|_{L^2},
	\ee
	where $a$ is some weakly continuous function of $u$ satisfying $a(0)=0$.
\end{lemma}

\begin{proof}
	We have
	\begin{align}
		&\int_0^1\left|\gut{v(\theta)}{\bt+g_1(s,s_0)\theta}-\gut{v(\theta)}{\bt+g_2(s,s_0)\theta}\right|d\theta\\&=\int_0^1\left|(\Put{\bt+g_1(s,s_0)\theta}{0,u}-\Put{\bt+g_2(s,s_0)\theta}{0,u})_*f_{v(\theta)}\right|d\theta\\&\leq\int_0^1\sum_{i=1}^l\left|v_i(\theta)\right|\left|(\Put{\bt+g_1(s,s_0)\theta}{0,u}-\Put{\bt+g_2(s,s_0)\theta}{0,u})_*X_i\right|d\theta\\&\leq\max_{i=1,\dotso,l}|X_i|\int_0^1\sum_{i=1}^l\left|v_i(\theta)\right|\left\|(\Put{\bt+g_1(s,s_0)\theta}{0,u}-\Put{\bt+g_2(s,s_0)\theta}{0,u})_*\right\|_\infty d\theta\\&\leq a(u)|g_1(s,s_0)-g_2(s,s_0)|\left(\int_0^1\left(\sum_{i=1}^l|v_i(\theta)|\right)^2d\theta\right)^{1/2}\\&\leq \sqrt{l}a(u)|g_1(s,s_0)-g_2(s,s_0)|\|v\|_{L^2},
	\end{align}
	where the second-last line follows by the Lipschitz continuity of the flow and the Cauchy-Schwartz inequality.
\end{proof}

\subsubsection{Expansion of the first-order term} We consider here the asymptotic expansion of the first-order term in \eqref{expansiong}. We have:
\begin{align}
	I'(s)=&|s+s_0|^{1/2}\int_0^1\gut{v(\theta)}{\bt+|s+s_0|^{3/4}\theta,u}d\theta-|s_0|^{1/2}\int_0^1\gut{v(\theta)}{\bt+|s_0|^{3/4}\theta,u}d\theta\\&=\underbrace{\left(|s+s_0|^{1/2}-|s_0|^{1/2}\right)\int_0^1\gut{v(\theta)}{\bt+|s+s_0|^{3/4}\theta,u}d\theta}_{A}-\underbrace{|s_0|^{1/2}\int_0^1\gut{v(\theta)}{\bt+|s_0|^{3/4}\theta,u}-\gut{v(\theta)}{\bt+|s+s_0|^{3/4}\theta,u}d\theta}_{B}.
\end{align}
If we now apply the conclusions of Lemma \ref{lemma:lipschitz}, to both the terms denoted by $A$ and $B$ (observe that in the first case we are implicitly using the fact that, as $v\in\mathcal{P}$, we have $\int_0^1\gut{v(\theta)}{\bt}dt=0$), we deduce that:
\begin{align}
	 \lim_{s\to 0}\frac{|A|}{s}&\leq a(u)\lim_{s\to 0}\frac{\left(|s+s_0|^{1/2}-|s_0|^{1/2}\right)}{s}|s+s_0|^{3/4}=\frac{1}{2}a(u)|s_0|^{1/4},\\
	 \lim_{s\to 0}\frac{|B|}{s}&\leq a(u)\lim_{s\to 0}\frac{\left(|s+s_0|^{3/4}-|s_0|^{3/4}\right)}{s}|s_0|^{1/2}=\frac{3}{4}a(u)|s_0|^{1/4}.
\end{align}

\subsubsection{Expansion of the second-order term} We turn now to the more complicated second order term in \eqref{expansiong}. Let us first recall the equality $w(t)=\int_0^t v(\tau)d\tau$; we have by construction that $w$ is a Lipschitz function from $I$ into $\R^l$. Then we recall two useful formulas \cite[Theorem 11.13]{AgrachevBarilariBoscain}:
\begin{align}
	\label{eq:intbyparts}
	\int_{t_1}^{t_2} \gut{v(t)}{t,u}dt&=\gut{w(t_2)}{t_2,u}-\gut{w(t_1)}{t_1,u}-\int_{t_1}^{t_2}\dgut{w(t)}{t,u}dt,\\
	\iint\limits_{t_1\leq\tau\leq t\leq t_2}[\gut{v(\tau)}{\tau,u},\gut{v(t)}{t,u}]d\tau dt&=\int_{t_1}^{t_2}[\int_{t_1}^{t}\gut{v(\tau)}{\tau,u}d\tau,\gut{v(t)}{t,u}]dt\\&=\int_{t_1}^{t_2}[\gut{v(\tau)}{\tau,u},\int_{\tau}^{t_2}\gut{v(t)}{t,u}dt]d\tau.
\end{align}

We start with the observation that:
\begin{align*} \iint\limits_{\bt\leq\tau\leq t\leq\bt+|q|^{3/4}}\gut{v_q(\tau)}{\tau,u}\circ \gut{v_q(t)}{t,u} d\tau dt&=\frac{1}{2}\int_{\bt}^{\bt+|q|^{3/4}}\gut{v_q(\tau)}{\tau,u} d\tau\circ\int_{\bt}^{\bt+|q|^{3/4}}\gut{v_q(t)}{t,u} dt\\&+\frac{1}{2}\iint\limits_{\bt\leq\tau\leq t\leq\bt+|q|^{3/4}}[\gut{v_q(\tau)}{\tau,u}, \gut{v_q(t)}{t,u}] d\tau dt.\end{align*}
Then (let us call for a moment $s'=s+s_0$)
\begin{align}
	I''(s)&=\frac{1}{2}\left(\int_{\bt}^{\bt+|s'|^{3/4}}\gut{v_{s'}(\tau)}{\tau,u} d\tau\circ\int_{\bt}^{\bt+|s'|^{3/4}}\gut{v_{s'}(t)}{t,u} dt+\iint\limits_{\bt\leq\tau\leq t\leq\bt+|s'|^{3/4}}[\gut{v_{s'}(\tau)}{\tau,u}, \gut{v_{s'}(t)}{t,u}] d\tau dt\right)\\&-\frac{1}{2}\left(\int_{\bt}^{\bt+|s_0|^{3/4}}\gut{v_{s_0}(\tau)}{\tau,u} d\tau\circ\int_{\bt}^{\bt+|s_0|^{3/4}}\gut{v_{s_0}(t)}{t,u} dt+\iint\limits_{\bt\leq\tau\leq t\leq\bt+|s_0|^{3/4}}[\gut{v_{s_0}(\tau)}{\tau,u}, \gut{v_{s_0}(t)}{t,u}] d\tau dt\right)
\end{align}
We take into considerations the terms without the commutator: if we add and subtract the term
\be
	\int_{\bt}^{\bt+|s'|^{3/4}}\gut{v_{s'}(\tau)}{\tau,u}d\tau\circ \int_{\bt}^{\bt+|s_0|^{3/4}}\gut{v_{s_0}(t)}{t,u}dt
\ee
we obtain the expression
\begin{align}
 C=\int_{\bt}^{\bt+|s'|^{3/4}}\gut{v_{s'}(\tau)}{\tau,u}d\tau&\left(\int_{\bt}^{\bt+|s'|^{3/4}}\gut{v_{s'}(t)}{t,u}dt-\int_{\bt}^{\bt+|s_0|^{3/4}}\gut{v_{s_0}(t)}{t,u}dt\right)\\&+\int_{\bt}^{\bt+|s_0|^{3/4}}\gut{v_{s_0}(t)}{t,u}dt\left(\int_{\bt}^{\bt+|s'|^{3/4}}\gut{v_{s'}(\tau)}{\tau,u}d\tau-\int_{\bt}^{\bt+|s_0|^{3/4}}\gut{v_{s_0}(\tau)}{\tau,u}dt\right);
\end{align}
using again Lemma \ref{lemma:lipschitz} and the same arguments as in the expansion of the first-order terms, we easily deduce that:
\be
	\lim_{s\to 0}\frac{|C|}{s}\leq \frac{3}{4}a(u)|s_0|^{1/2}.
\ee

We move now to the commutator term: again we begin with a general computation, namely using \eqref{eq:intbyparts} and the fact that $w_q(\bt)=w_q(\bt+|q|^{3/4})=0$, we can write:
\begin{align}
	\int_{\bt}^{\bt+|q|^{3/4}}[\int_{\bt}^{t}\gut{v_q(\tau)}{\tau,u}d\tau,\gut{v_q(t)}{t,u}]dt&=\int_{\bt}^{\bt+|q|^{3/4}}[\gut{w_q(t)}{t,u},\gut{v_q(t)}{t,u}]dt-\int_{\bt}^{\bt+|q|^{3/4}}[\int_{\bt}^{t}\dgut{w_q(\tau)}{\tau,u} d\tau,\gut{v_q(t)}{t,u}]dt\\&=\int_{\bt}^{\bt+|q|^{3/4}}[\gut{w_q(t)}{t,u},\gut{v_q(t)}{t,u}]dt-\int_{\bt}^{\bt+|q|^{3/4}}[\dgut{w_q(\tau)}{\tau,u},\int_{\tau}^{\bt+|q|^{3/4}}\gut{v_q(t)}{t,u}dt]d\tau\\&=\int_{\bt}^{\bt+|q|^{3/4}}[\gut{w_q(t)}{t,u},\gut{v_q(t)}{t,u}]dt+\int_{\bt}^{\bt+|q|^{3/4}}[\dgut{w_q(\tau)}{\tau,u},\gut{w_q(\tau)}{\tau,u}]d\tau\\&+\int_{\bt}^{\bt+|q|^{3/4}}[\dgut{w_q(\tau)}{\tau,u},\int_{\tau}^{\bt+|q|^{3/4}}\dgut{w_q(t)}{t,u}dt]d\tau.
\end{align}
We immediately realize that just the first summand matters for our purposes. Indeed in both the other summands there is at least a double change of variables $\theta=\frac{t-\bt}{|q|^{3/4}}$ plus a differentiation, which yields a power of $|q|$ not smaller than $7/4$. By virtue of the equalities
\begin{align}
	v_q(\bt+|q|^{3/4}\theta)&=\frac{1}{|q|^{1/4}} v(\theta),\\
	w_q(\bt+|q|^{3/4}\theta)&=\int_{\bt}^{\bt+|q|^{3/4}\theta}\frac{1}{|q|^{1/4}} v\left(\frac{\tau-\bt}{|q|^{3/4}}\right)d\tau\\&=|q|^{1/2}\int_0^\theta v(\zeta)d\zeta=|q|^{1/2}w(\theta),
\end{align}
we have
\begin{align}
	\int_{\bt}^{\bt+|q|^{3/4}}[\gut{w_q(t)}{t,u},&\gut{v_q(t)}{t,u}]dt=\\&=|q|^{3/4}\int_0^1[\gut{w_q(\bt+|q|^{3/4}\theta)}{\bt+|q|^{3/4}\theta,u},\gut{v_q(\bt+|q|^{3/4}\theta)}{\bt+|q|^{3/4}\theta,u}]d\theta=|q|\int_0^1[\gut{w(\theta)}{\bt+|q|^{3/4}\theta,u},\gut{v(\theta)}{\bt+|q|^{3/4}\theta,u}].
\end{align}
Call
\be
	D=|s'|\int_0^1[\gut{w(\theta)}{\bt+|s'|^{3/4}\theta,u},\gut{v(\theta)}{\bt+|s'|^{3/4}\theta,u}]d\theta-|s_0|\int_0^1[\gut{w(\theta)}{\bt+|s_0|^{3/4}\theta,u},\gut{v(\theta)}{\bt+|s_0|^{3/4}\theta,u}]d\theta.
\ee
Adding and subtracing the common term
\be
	|s'|\int_0^1[\gut{w(\theta)}{\bt+|s_0|^{3/4}\theta,u},\gut{v(\theta)}{\bt+|s_0|^{3/4}\theta,u}]d\theta,
\ee
we end up with
\begin{align}
	D&=\underbrace{|s'|\int_0^1[\gut{w(\theta)}{\bt+|s'|^{3/4}\theta,u},\gut{v(\theta)}{\bt+|s'|^{3/4}\theta,u}]-[\gut{w(\theta)}{\bt+|s_0|^{3/4}\theta,u},\gut{v(\theta)}{\bt+|s_0|^{3/4}\theta,u}]d\theta}_{L}\\&+\underbrace{(|s'|-|s_0|)\int_0^1[\gut{w(\theta)}{\bt+|s_0|^{3/4}\theta,u},\gut{v(\theta)}{\bt+|s_0|^{3/4}\theta,u}]d\theta}_{H}.
\end{align}
On the one hand, by Lemma \ref{lemma:lipschitz}, there holds
\begin{align}
	L&=|s'|\int_0^1[\gut{w(\theta)}{\bt+|s'|^{3/4}\theta,u},\gut{v(\theta)}{\bt+|s'|^{3/4}\theta,u}-\gut{v(\theta)}{\bt+|s_0|^{3/4}\theta,u}]+[\gut{w(\theta)}{\bt+|s'|^{3/4}\theta,u}-\gut{w(\theta)}{\bt+|s_0|^{3/4}\theta,u},\gut{v(\theta)}{\bt+|s_0|^{3/4}\theta,u}]d\theta\\&\leq a(u)|s'|(|s'|^{3/4}-|s_0|^{3/4}),
\end{align}
which implies that
\be
	\lim_{s\to 0}\frac{|L|}{s}\leq \frac{3}{4}a(u)|s_0|^{3/4};
\ee
on the other hand, whenever $s_0\neq 0$, we have by similar reasonings (notice that $\|v\|_{L^2}\leq 1$ implies that also $\|w\|_{L^2}\leq 1$):
\be
	\lim_{s\to 0}\frac{H}{s}=\textrm{sgn}(s_0)\int_0^1[\gut{w(\theta)}{\bt,u},\gut{v(\theta)}{\bt,u}]d\theta+R(u,s_0),\quad\textrm{with }|R(u,s_0)|\leq a(u)|s_0|^{3/4},
\ee
while if $s_0=0$ the following is true:
\be
		\limsup_{s\to 0}\frac{H}{s}=\int_0^1[\gut{w(\theta)}{\bt,u},\gut{v(\theta)}{\bt,u}]d\theta,\quad \liminf_{s\to 0}\frac{H}{s}=-\int_0^1[\gut{w(\theta)}{\bt,u},\gut{v(\theta)}{\bt,u}]d\theta
\ee
Collecting all the above calculations we finally obtain
\begin{prop}\label{propo:estimate}
For every $v\in\mathcal P$, the following estimates holds
	\begin{align}
		&\lim_{s\to 0}\frac{F(u+v_{s+s_0})-F(u+v_{s_0})}{s}= \frac{\textrm{sgn}(s_0)}{2}\int_0^1(\Put{0}{1,u})_*[\gut{w(\theta)}{\bt,u},\gut{v(\theta)}{\bt,u}]d\theta+R(u,s_0),\\&\limsup_{s\to 0}\frac{F(u+v_{s})-F(u)}{s}=\int_0^1(\Put{0}{1,u})_*[\gut{w(\theta)}{\bt,u},\gut{v(\theta)}{\bt,u}]d\theta,\\&\liminf_{s\to 0}\frac{F(u+v_{s})-F(u)}{s}=-\int_0^1(\Put{0}{1,u})_*[\gut{w(\theta)}{\bt,u},\gut{v(\theta)}{\bt,u}]d\theta.
	\end{align}
	Moreover, the map $(u,s_0)\mapsto R(u,s_0)$ is weakly continuous with respect to $u$, satisfies the equality $R(u,0)=0$ for every $u$ and, for $|s_0|$ sufficiently small, there holds the estimate
	\be
		|R(u,s_0)|\leq a(u)(|s_0|^{1/4}).
	\ee
\end{prop}
\begin{lemma}\label{lemma:inverse}For every soft $u_0\in \mathcal{A}^E$ there exists a weak neighborhood $\mathcal{W}_2$ of $u_0$, a neighborhood $I_2\subset \R$ of zero, a neighborhood $W_2\subset \R^m$ of zero and a positive number $r_2'>0$ such that for every $(u,\psi)\in \mathcal{W}_2\times W_2$ the function:
\be x\mapsto \varphi_0\left(u,x,\psi\right)-\varphi_0(u,0,0), \quad x\in I_2\ee
is invertible with inverse defined on $(-r_2', r_2')$ and which depends continuously on $u$ and $\psi$. More precisely, there exists:
\be g_2:\mathcal{W}_2\times (-r_2', r_2')\times W_2\to I_2\ee
which is continuous for the strong topology and such that:
\be \textrm{$g_2(u,s,\psi)$ is the unique solution to  $\varphi_0\left(u,g_2(u,s,\psi),\psi\right)-\varphi_0(u,0,0)=s$}\ee
Moreover there exist $r_2''>0$ and a weak neighborhood $\mathcal{V}_2\subset \mathcal{W}_2$ such that $B(u, r_2'')\subset \mathcal{W}_2$ for every $u\in \mathcal{V}_2\cap \{J\leq E\}$.
\end{lemma}

\begin{proof}
	Consider the map $G_0(u,x,\psi)=\varphi_0(u,x,\psi)-\varphi_0(u,0,0)$. We know from Proposition \ref{propo:estimate} that $G_0(u,x,0)$ is both weakly continuous in the $u$-variable and Lipschitz continuous in the $x$-variable. Let us define: \be H_0(u,x,\psi)=\min \left|\frac{\partial G_0}{\partial x}\bigg|_{(u,x,\psi)}\right|\ee where the min is taken over all the elements in the Clarke $x$-subderivative of $G_0.$ Our choice of $v^{\pm}$ ensures that both
	\be
		\liminf_{x\to 0^{\pm}}\frac{\partial G_0(u,x,0)}{\partial x}=\pm\int_0^1\langle(\Put{\bt}{1,u_0})^*\lambda,[f_{w^{\pm}(\theta)},f_{v^{\pm}(\theta)}]\rangle d\theta>0;
	\ee
	in particular the subdifferential $H(u,x,0)$ is not zero. By the Clarke's Implicit Function Theorem \cite{Clarke} and the weak continuity of $u\mapsto H_0(u,x,0)$, we deduce that there exist a weak neighborhood $\mathcal{W}_2'$ of $u_0$ and $r>0$ such that $G_0(u,x,0):(-r,r)\to \R$ is an homeomorphism onto its image, for every $u\in\mathcal{W}_2'$.
	
	Consider the map $(y_1,\dotso,y_m)\mapsto y_1e_1(u_0)+\cdots+y_me_m(u_0)$: being continuous, there exist a neighborhood $W_2'\subset\R^m$ of zero and a weak neighborhood $\mathcal{W}_2''$ of $u_0$ such that $u+y_1e_1(u_0)+\cdots +y_me_m(u_0)$ belongs to $\mathcal{W}_2'$ for every $(u,y)\in \mathcal{W}_2''\times W_2'$. Finally, since the map $(\psi_1,\dotso,\psi_m)\mapsto (y_1,\dotso,y_m)$ is continuous as well, we conclude that there exists also a neighborhood $W_2''\subset\R^m$ of zero such that $H_0(u,x,\psi)$ is of maximal rank whenever $(u,\psi)\in\mathcal{W}_2''\times W_2''$.
	
	We take advantage of this fact to show that, for every $(u,\psi)\in\mathcal{W}_2''\times W_2''$, the map $G_0(u,\cdot,\psi):(-r,r)\to \R$ is injective. Indeed, let $a$ be any positive number such that $H(u,x,\psi)>a$ on $\mathcal{W}_2''\times(-r,r)\times W_2''$; then, by the Lipschitz version of the mean value theorem \cite[Theorem 2.6.5]{Clarke2}, we infer that: \be|G_0(u,x_1,\psi)-G_0(u,x_2,\psi)|>a|x_1-x_2|.\ee Again we can use \cite[Lemma 5]{Clarke} to conclude that, whenever $r'<r$, then: \be G_0(u,\cdot,\psi)(-r',r')\supset G_0(u,0,\psi)+(-r'a,r'a);\ee moreover, the weak continuity of $u\mapsto G_0(u,x,\psi)$, and the fact that $G_0(u,0,0)=0$, imply that for every $r>0$, there exists a weak open set $\mathcal{W}_2^{\textrm{weak}}(r)\times W_2(r)$ on which $|G_0(u,0,\psi)|<ar'/3$. Then, if $r'<r$, for every $(u,\psi)\in(\mathcal{W}_2''\cap\mathcal{W}_2^{\textrm{weak}}(r'))\times(W_2''\cap W_2(r'))$ we will have that:
	\be
		G_0(u,\cdot,\psi)\supset G_0(u,0,\psi)+(-r'a,r'a)\supset (-2/3r'a,2/3r'a).
	\ee
	Choose $0<\widetilde{r}<r/2$, and let $\mathcal{W}_3^{\textrm{weak}}=\mathcal{W}_2''\cap\mathcal{W}_2^{\textrm{weak}}(\widetilde{r})$ and $W_3=W_2''\cap W_2(\widetilde{r})$; our previous arguments then show that, for every $(u,\psi)\in\mathcal{W}_3^{\textrm{weak}}\times W_3$, the map \be G_0(u,\cdot,\psi):(-\widetilde{r},\widetilde{r})\to \R \ee is an homeomorphism onto its image, and its image contains $(-2/3\widetilde{r}a,2/3\widetilde{r}a)$.
	
	Similarly as in Lemma \ref{lemma:coord}, we define the weakly continuous functions:
	\be
		\alpha_1(u)=\min_{(x,\psi)\in (-\widetilde{r},\widetilde{r})\times \overline{W}_3} H_0(u,x,\psi),\quad\textrm{and }\quad \alpha_2(u)=\max_{\psi\in\overline{W}_3}|G_0(u,0,\psi)|
	\ee
	and, for $\epsilon>0$ small enough, the weakly open sets:
	\be \mathcal{W}_2=\{\alpha_1>a, \alpha_2<a\tilde r/3\},\quad \mathcal{V}_2=\{\alpha_1>a+\epsilon, \alpha_2<a\tilde r/3-\epsilon\}, \quad W_2=W_3\quad I_2=(-\widetilde{r},\widetilde{r}). \ee
	Then we conclude as in Lemma \ref{lemma:coord}, by choosing e.g. $r_2'=2/3\widetilde{r}a$, and where the existence of $r_2''$ is guaranteed by applying Lemma \ref{lemma:wdist} to the sets $(\mathcal{W}_2)^c$ and $\overline{\mathcal{V}}_2\cap\{J\leq E\}$. Finally, the assertion on the existence and the continuity of the function $g_2$ follows literally as in Corollary \ref{cor:impl1}.
\end{proof}

\subsection{Proof of Proposition \ref{prop:proof}}\label{sec:proof}
Let $u_0\in \mathcal{A}^E$ be soft and define $\mathcal{W}=\mathcal{W}_1\cap \mathcal{W}_2$ and $\mathcal{V}=\mathcal{V}_1\cap \mathcal{V}_2$. Let $r_2=\min\{r_1'', r_2''\}$ and $r_1, r_3>0$ be such that:
\be B(0, r_1)\subset \left(I_1\cap(-r_2', r_2') \right)\times \left(B(0, r_1')\cap W_2\right)\quad\textrm{and}\quad B(0, r_3)\supset I_2\times B_1.\ee
(All these objects have been constructed in Lemma \ref{lemma:coord} and Lemma \ref{lemma:inverse}.)
 Then the function $g$ defined by:
\be g(x, y, u)=(g_2(x,u), g_1(g_2(x,u), y, u))\ee
verifies the required properties.

	\section{An implicit function theorem}\label{sec:implicit}
\begin{thm}\label{thm:implicit}Assume all abnormal controls with $J\leq E$ are soft. There exists a neighborhood $\mathcal{W}(\mathcal{A}^E)$ and positive numbers $r_1, r_2, r_3>0$ such that for every $u_0\in \mathcal{W}(\mathcal{A}^E)$ there exists a function:
\be \sigma_{u_0}:B(0, r_1)\times B(u_0, r_2)\to L^{2}(I, \R^l)\ee
which is continuous for the strong topology and such that:
\be \sigma_{u_0}(0, u)=0\quad \textrm{and}\quad \varphi(u+\sigma_{u_0}(w,u))=\varphi(u)+w\quad \forall (w,u)\in B(0, r_1)\times B(u_0, r_2).\ee
Moreover the family $\{\sigma_{u_0}\}_{u_0\in \mathcal{W}(\mathcal{A}^E)}$ is equicontinuous.
\end{thm}
\begin{proof}
Every $u\in \mathcal{A}^E$ is soft and we can consider the weak open sets $\mathcal{V}(u)\subset\mathcal{W}(u)$, the positive numbers $r_1(u), r_2(u)>0$ and the functions $g=g_u$ and $\alpha_u$ constructed in Section \ref{sec:cross}.
Notice that $\mathcal{A}^E$ is weakly compact (it is a weakly closed set in the weakly closed ball $\{J\leq E\}$). Then $\{\mathcal{V}(u)\}_{u\in \mathcal{A^E}}$ is a weak cover of $\mathcal{A}^E$ and consequently there exist $u_1, \ldots, u_k\in \mathcal{A}^E$ such that:
\be \mathcal{W}(\mathcal{A}^E)=\mathcal{V}_{u_1}\cup\cdots\cup \mathcal{V}_{u_k}\ee
is an open neighborhood of $\mathcal{A}^E.$ Set $r_1=\min_{i}\{r_1(u_i)\}$ and $r_2=\min_{i}\{r_2(u_i)\}$.

Pick $u\in \mathcal{W}(\mathcal{A}^E)$. Then $u\in \mathcal{W}_{u_i}$ for some $i\in\{1, \ldots, k\}$ and we define the function $\sigma_u$ by:
\be \sigma_u(w, v)=\alpha_{u_i}(g_{u_i}(w, v))\quad \textrm{for $(w, v)\in B(0, r_1)\times B(u, r_2)$}.\ee
Notice that $g_{u_i}$ was defined on $B(0, r_1(u_i))\times \mathcal{W}_{u_i}$; on the other hand since $u\in \mathcal{V}_{u_i}$ then $B(u, r_2)\subset \mathcal{W}_{u_i}$ and the domain of $\sigma_u$ is contained in the domain of definition of $g_{u_i}.$

The family $\{\sigma_{u}\}_{u\in \mathcal{W}(\mathcal{A}^E)}$ is equicontinuous simply because it is finite.
\end{proof}	
	
		\section{Regular controls: gradient flow}\label{sec:regular}

Throughout this section we will assume that all abnormal controls with $J\leq E$ are soft.

\begin{remark}
	As the endpoint map is weakly continuous, the set $F^{-1}(y)$ is weakly closed. Then $\mathcal{U}(y)^E=\{u\in\mathcal{U}\,|\,J(u)\leq E\}\cap F^{-1}(y)$ is weakly compact in $\mathcal{U}^E$, this latter set also being weakly compact.
\end{remark}

Let $\mathcal{W}(\mathcal{A}^E)$ be the weak neighborhood of the abnormal controls with energy less than $E$ constructed in the previous section (abusing of the notation, we tacitly adopt the convention that $\mathcal{W}(\mathcal{A}^E)$ and $\mathcal{A}^E$ are to be intended in the relative topology of $\mathcal{U}(y)^E$; they are, respectively, weakly open and weakly compact), and let $\mathcal{B}^E=\mathcal{U}(y)^E\setminus \mathcal{W}(\mathcal{A}^E)$. Now, $\mathcal{B}^E$ and $\mathcal{A}^E$ are two disjoint weakly compact subsets in $\mathcal{U}(y)^E$, therefore they can be separated by means of weak neighborhoods $\mathcal{V}(\mathcal{A}^E)$ and $\mathcal{V}(\mathcal{B}^E)$. In particular we have the following sequence of inclusions:
\be
	\label{eq:inclusions}
		\mathcal{B}^E\subset\mathcal{V}(\mathcal{B}^E)\subset \mathcal{U}(y)^E\setminus\mathcal{V}(\mathcal{A}^E),
\ee
and, by Lemma \ref{lemma:wdist}, the weakly compact set $\mathcal{U}(y)^E\setminus\mathcal{V}(\mathcal{A}^E)$ is strongly separated from $\mathcal{A}^E$. Exploiting this distance from the abnormal set, the idea is to mimic the classical deformation theory via the gradient flow. If we call
\be
	f=J\big|_{\mathcal{U}(y)^E\setminus\mathcal{V}(\mathcal{A}^E)}
\ee
the restriction of the energy, we can apply verbatim the arguments of \cite[Proposition 10]{BoarottoLerario} in this case: the only salient fact to be observed is that the set $\mathcal{U}(y)^E\setminus\mathcal{V}(\mathcal{A}^E)$ is weakly closed in $\mathcal{U}(y)^E$, and, as such, it contains all the weak limits of its sequences (in particular they are all regular points of the endpoint map). Having this in mind, we obtain the following.

\begin{prop}[Palais-Smale condition]
	\label{prop:PS}
	The function $f$ satisfies the Palais-Smale condition, i.e any sequence $\{u_n\}_{n\in \N}\subset\mathcal{U}(y)^E\setminus\mathcal{V}(\mathcal{A}^E)$ such that
	\be
		\lim_{n\to\infty}d_{u_n}f\to 0
	\ee
	admits a convergent subsequence.
\end{prop}

As an immediate consequence we also recover the following standard lemma.

\begin{cor}[Palais-Smale lemma]
	\label{cor:PSL}
	Let $0<s<E$, and let $\mathcal{N}$ be any open set contained in $\mathcal{U}(y)^E$ (possibly empty). Assume that
	\be
	(\mathcal{U}(y)^E\setminus(\mathcal{V}(\mathcal{A}^E)\cup \mathcal{N}))\cap\{J=s\}\cap\mathcal{C}=\emptyset;
	\ee
	then there exists a positive constant $0<\eta(s)<1$ such that
	\be
		\|d_uf\|>\eta,\quad\forall u\in(\mathcal{U}(y)^E\setminus(\mathcal{V}(\mathcal{A}^E)\cup \mathcal{N}))\cap\{s-\eta< J< s+\eta\}.
	\ee
\end{cor}

\begin{proof}
	We argue by contradiction and assume $\{u_n\}_{n\in\N}$ to be a sequence in $\mathcal{U}(y)^E\setminus(\mathcal{V}(\mathcal{A}^E)\cup \mathcal{N})$ such that
	\begin{itemize}
		\item [1.] $u_n\in\{s-1/n<J<s+1/n\}$
		\item [2.] $\|d_{u_n}f\|<1/n$.
	\end{itemize}
	Then the assumptions of Proposition \ref{prop:PS} are satisfied and therefore, passing possibly to a subsequence, we may assume that $\overline{u}=\lim_{n}u_n$ exists in $\mathcal{U}(y)^E\setminus(\mathcal{V}(\mathcal{A}^E)\cup \mathcal{N})$, as this set is closed. By point 1., $J(\overline{u})=\lim_n J(u_n)=s$; moreover, by point 2., $\overline{u}$ has to be a critical point of $f$, which leads to an absurd.
\end{proof}

A further application of Lemma \ref{lemma:wdist} gives $\beta>0$ such that
\be
\label{eq:distt}
\textrm{dist}\left(\mathcal{B}^E,\mathcal{U}(y)^E\setminus\mathcal{V}(\mathcal{B}^E)\right)>\beta;
\ee
in this situation we define the (strongly) open set
\be
	\label{eq:W}\mathcal{W}(\mathcal{B}^E)=\bigcup_{u\in\mathcal{B}^E}B\left(u,\frac{\beta}{2}\right)
\ee
as the union of open balls of radius $\beta/2$, centered on elements of $\mathcal{B}^E$.

The next proposition is a refinement of a classical result, which can for instance be found in \cite[Chapter 1, Theorem 3.4]{Chang}, adapted to our setting. Notice that this statement is stronger than the analogous statement as in \cite[Chapter 1, Theorem 3.3]{Chang}, which is in fact a consequence of \cite[Chapter 1, Theorem 3.4]{Chang} (the constants depend on the Palais-Smale condition).

\begin{prop}
	\label{prop:def}
		Let $0<s<E$, $0<\delta<\min\{\beta,1\}$ and $\mathcal{N}'\subset \mathcal{N}$ be two open sets in $\mathcal{U}(y)^E$ such that $\textrm{dist}(\overline{\mathcal{N}'},\mathcal{U}(y)^E\setminus\mathcal{N})>\delta$. Assume that there exists $0<\eta<1$ such that
		\be
		\|d_u f\|> \eta,\quad\,\forall u\in(\mathcal{U}(y)^E\setminus(\mathcal{V}(\mathcal{A}^E)\cup \mathcal{N}'))\cap\{s-\eta<J<s+\eta\},
		\ee
		and let
		\be
				t=\eta\frac{\beta}{2},\quad\,0<\varepsilon'<\eta\frac{\delta}{4},\quad\textrm{and }\,\varepsilon'<\varepsilon''<\eta.
		\ee
		Then there exist a continuous map \be\Theta:[0,t]\times \mathcal{W}(\mathcal{B}^E)\to\mathcal{U}(y)^E\setminus\mathcal{V}(\mathcal{A}^E),\ee such that
		\begin{itemize}\setlength\itemsep{0.3em}
			\item [a)] $\Theta(0,\cdot)=\textrm{Id}$,
			\item [b)] $\Theta(\tau,\cdot)=\textrm{Id}$ on the set $\{J\leq s-\varepsilon''\}\cup \{J\geq s+\varepsilon''\}$,
			\item [c)] For every $u\in(\mathcal{W}(\mathcal{B}^E)\setminus \mathcal{N})\cap\{J\leq s+\varepsilon'\}$, $\Theta(t,u)\in \mathcal{U}(y)^E\setminus\mathcal{V}(\mathcal{A}^E)\cap\{J\leq s-\varepsilon'\}$,
			\item [d)] $J(\Theta(\tau,u))$ is nonincreasing, for any $(\tau,u)\in[0,t]\times \mathcal{U}(y)^E\setminus\mathcal{V}(\mathcal{A}^E)$.
		\end{itemize}
\end{prop}

\begin{proof}
	Define
	\be
		\chi(r)=\left\{\begin{aligned}
						0 & \quad \textrm{if } r\not\in(s-\varepsilon'',s+\varepsilon'')\\
						1 & \quad \textrm{if } r \in [s-\varepsilon',s+\varepsilon']
					\end{aligned}
			 \right.
	\ee
	to be a smooth function satisfying $0\leq \chi(r)\leq 1$.  Consider two closed subsets $C_1=\mathcal{U}(y)^E\setminus(\mathcal{V}(\mathcal{A}^E)\cup \mathcal{N}'_{\delta/2})$, where $\mathcal{N}'_{p}=\{u\in \mathcal{U}(y)^E\,|\,\textrm{dist}(u,\overline{\mathcal{N}'})<p\}$, and $C_2=\overline{\mathcal{N}'}\cap(\mathcal{U}(y)^E\setminus \mathcal{V}(\mathcal{A}^E))$. Then we can construct the function
	\be
		g(u)=\frac{\textrm{dist}(u,C_2)}{\textrm{dist}(u,C_1)+\textrm{dist}(u,C_2)},
	\ee
	so that $0\leq g(u)\leq 1$, $g\equiv 1$ in $C_1$ and $g\equiv 0$ in $C_2$. On the set $\mathcal{U}(y)^E\setminus\mathcal{V}(\mathcal{A}^E)$, we define the vector field
	\be
		\label{eq:vectorfield}
			Y(u)=-g(u)\chi(J(u))\frac{d_uf}{\|d_uf\|^2}.
	\ee
	By the standard theory of differential equations on the real line, since $\|Y\|< \frac{1}{\eta}$, it is well-defined the time-$x$ flow of $Y$ starting from the point $u_0$ for any time $x>0$, which we will indicate by $\psi_x^Y(u_0)$. Pick $u\in\mathcal{W}(\mathcal{B}^E)\setminus \mathcal{N}$, and assume that it belongs to a ball centered at $u_0$. If we let $\psi^Y_\cdot(u)$ flow for a time $T_1\leq t$, by virtue of \eqref{eq:distt} and the inequality
	\be
	\|\psi_{T_1}^Y(u)-u\|\leq \int_0^{T_1}\|Y(\psi_\tau^Y(u))\|d\tau<\frac{\beta}{2},
	\ee
	we see that:
	\be
		\|\psi_{T_1}^Y(u)-u_0\|\leq \|u-u_0\|+\|\psi_{T_1}^Y(u)-u\|<\frac{\beta}{2}+\frac{\beta}{2}=\beta,
	\ee
	that is $\psi_{T_1}^Y(u)$ belongs to $\mathcal{V}(\mathcal{B}^E)\subset\mathcal{U}(y)^E\setminus\mathcal{V}(\mathcal{A}^E)$. Then we define the deformation map $\Theta$ by \be
		\Theta(\tau,u)=\psi_\tau^Y(u),\quad\forall\,(\tau,u)\in[0,t]\times\mathcal{W}(\mathcal{B}^E).
	\ee
	
	Points a), b) and d) are almost immediate: the only non trivial point to verify is c). Let $T_2=\frac{\delta}{2}\eta<\frac{\beta}{2}\eta=t$; we claim that flowing for time $T_2$ suffices for our purposes, that is we want to show that
	\be
	J(\psi_{T_2}^Y(u))\leq s-\varepsilon',\quad \forall\,u\in (\mathcal{W}(\mathcal{B}^E)\setminus \mathcal{N})\cap \{s-\varepsilon'<J\leq s+\varepsilon'\}.
	\ee
	We observe at first that
	\be
		\label{eq:one}
		\frac{d}{dt}J(\psi_\tau^Y(u))=-g(\psi_\tau^Y(u))\chi(J(\psi_\tau^Y(u))),
	\ee
	To prove our claim we argue by contradiction, and we assume
	\be\label{eq:absurd}
		s-\varepsilon'<J(\psi_{\tau}^Y(u))\leq s+\varepsilon',\quad\forall\,\tau\in\left[0,\frac{\delta}{2}\eta\right],
	\ee
	so that $\chi(J(\psi_{\tau}^Y(u)))\equiv 1$. Moreover, since $\|\psi_\tau^Y(u)-u\|<\frac{\tau}{\eta}$, we also have
	\begin{align}
		\textrm{dist}(\psi_\tau^Y(u),\mathcal{N}'_{\delta/2})&> \textrm{dist}(u,\mathcal{N}'_{\delta/2})-\frac{\tau}{\eta}\\
		&>\delta-\frac{1}{2}\delta-\frac{1}{\eta}\frac{\eta\delta}{2}=0,
	\end{align}
	so that we even have the equality $g(\psi_{\tau}^Y(u))\equiv 1$. Finally, it follows from \eqref{eq:one} that
	\be
		\frac{d}{dt}J(\psi_\tau^Y(u))=-1,
	\ee
	which implies, combined with our choice of $\varepsilon'<\frac{\delta}{4}\eta$, that the following line is true
	\be
		J(\psi_{T_2}^Y(u))= J(u)-\frac{\delta}{2}\eta\leq s+\varepsilon'-\frac{\delta}{2}\eta<s-\varepsilon'.
	\ee
	Since this contradicts \eqref{eq:absurd}, the proof is complete.
\end{proof}

\begin{cor}
	\label{cor:normgeod}
	Let $\mathcal{N}$ be a neighborhood of
	\be
		\mathcal{C}_s=(\mathcal{U}^E\setminus\mathcal{V}(\mathcal{A}^E))\cap\{J=s\}\cap\mathcal{C}.
	\ee
	Then there exist $0<\eta<1$ and a deformation map $\Theta$ satisfying the conclusions of Proposition \ref{prop:def}.
\end{cor}

\begin{proof}
	The Palais Smale condition implies that the set $\mathcal{C}_s$ is (sequentially) compact; therefore for sufficiently small values of the parameter $\nu>0$ the closure of the set
	\be
		\mathcal{C}_s(\nu)=\{u\in \mathcal{U}^E\,|\,\textrm{dist}(u,\mathcal{C}_s)<\nu\}
	\ee
	is contained in $\mathcal{N}$. Let $\overline{\nu}$ be such that the inclusion holds, and define $\mathcal{N}'=\mathcal{C}_s(\overline{\nu})$. Then by construction
	\be
		\mathcal{U}(y)^E\setminus(\mathcal{V}(\mathcal{A}^E)\cup \mathcal{N}')\cap\{J=s\}\cap\mathcal{C}=\emptyset;
	\ee
	therefore, by Corollary \ref{cor:PSL}, there exists $0<\eta<1$ satisfying the assumptions of Proposition \ref{prop:def}, and then we conclude.
\end{proof}

Using somewhat the same ideas as in Proposition \ref{prop:def} we also have the following proposition.

\begin{prop}
	\label{prop:defuniform}
	Let $0<E_1<E_2<E$ and $\epsilon>0$. Assume that there exists $0<\eta<\min\{\epsilon/2,1\}$ such that
	\be
	\|d_u f\|> \eta,\quad\,\forall u\in(\mathcal{U}(y)^E\setminus\mathcal{V}(\mathcal{A}^E))\cap\{E_1+\epsilon-\eta<J<E_2+\eta\},
	\ee
	and let
	\be
	t=\eta\frac{\beta}{2}, \quad\textrm{and }\,0<\varepsilon'<\varepsilon''<\eta.
	\ee
	Then there exist a continuous map
	\be
		\Theta:[0,t]\times \mathcal{W}(\mathcal{B}^E)\to\mathcal{U}(y)^E\setminus\mathcal{V}(\mathcal{A}^E),
	\ee
	such that
	\begin{itemize}\setlength\itemsep{0.3em}
		\item [a)] $\Theta(0,\cdot)=\textrm{Id}$,
		\item [b)] $\Theta(\tau,\cdot)=\textrm{Id}$ on the set $\{J\leq E_1+\epsilon-\varepsilon''\}\cup \{J\geq E_2+\varepsilon''\}$,
		\item [c)] For any $u\in\mathcal{W}(\mathcal{B}^E)\cap\{E_1+\epsilon\leq J\leq E_2\}$ and $0\leq \tau<\varepsilon'$, $\Theta(\tau,u)\in\mathcal{U}(y)^E\setminus\mathcal{V}(\mathcal{A}^E)$ and $J(\Theta(\tau,u))=J(u)-\tau$.
	\end{itemize}
\end{prop}

\begin{proof}[Sketch of proof]
	In this case (the notations are mutuated from Proposition \ref{prop:def}), we have
	\be
	\chi(r)=\left\{\begin{aligned}
		0 & \quad \textrm{if } r\not\in(E_1+\epsilon-\varepsilon'',E_2+\varepsilon'')\\
		1 & \quad \textrm{if } r \in [E_1+\epsilon-\varepsilon',E_2+\varepsilon']
	\end{aligned}
	\right.
	\ee
	and
	\be
		\label{eq:vectorfield1}
		Y(u)=-\chi(J(u))\frac{d_uf}{\|d_uf\|^2}.
	\ee
	The flow up to time $t$ is well-defined, since for any point $u$ in $\mathcal{W}(\mathcal{B}^E)$, $\psi_\tau^Y(u)$ stays within $\mathcal{U}(y)^E\setminus\mathcal{V}(\mathcal{A}^E)$, and we define the deformation map $\Theta$ by \be
	\Theta(\tau,u)=\psi_\tau^Y(u),\quad\forall\,(\tau,u)\in[0,t]\times\mathcal{W}(\mathcal{B}^E).
	\ee
	Only point c) needs a verification: but, since for any $u\in\mathcal{W}(\mathcal{B}^E)\cap\{E_1+\epsilon\leq J\leq E_2\}$ and $0\leq\tau<\varepsilon'$
	\be
		J(\psi_\tau^Y(u))\geq J(u)-\varepsilon'\geq E_1+\epsilon-\varepsilon',
	\ee
	then $\chi(J(\psi_\tau^Y(u)))\equiv 1$, and the claim follows.
\end{proof}

\begin{cor}
	\label{cor:cordefuniform}
	Let $0<E_1<E_2<E$ be such that
	\be
		(\mathcal{U}(y)^E\setminus\mathcal{V}(\mathcal{A}^E))\cap\{E_1<J\leq E_2\}\cap\mathcal{C}=\emptyset.
	\ee
	Then, for any $\epsilon>0$, the conclusions of Proposition \ref{prop:defuniform} hold on the strip $\{E_1+\epsilon\leq J\leq E_2\}$.
\end{cor}

\begin{proof}
	Let $\epsilon>0$ be fixed. Then, for any $E_1+\epsilon\leq s\leq E_2$, Corollary \ref{cor:PSL} applies, and permits to find the corresponding parameter $0<\eta(s)<1$. Since
	\be
		[E_1+\epsilon,E_2]\subset \bigcup_{s\in [E_1+\epsilon,E_2]}\underbrace{(s-\eta(s),s+\eta(s))}_{I(s)},
	\ee
	by compactness we may extract a finite subcover consisting of the open intervals $I(s_1),\dotso,I(s_p)$. Now it is sufficient to choose $\eta=\min\{\eta_1,\dotso,\eta_p,\epsilon/2\}$ (notice in particular that $0<\eta<\min\{1,\epsilon/2\}$) to see that the assumption of Proposition \ref{prop:defuniform} are satisfied.
\end{proof}

	\section{The Serre fibration property}\label{sec:serre}
  We will need the following preliminary lemma, asserting that the property of being a Serre fibration can be verified locally. Notice that if the open cover  in the statement of Lemma \ref{lemma:localserre}  did not depend on $n$, then this is classical (see \cite{Hurewicz} for a proof in the more general case of a Hurewicz fibration); instead here we have to work with an open cover that might depend on $n$, but in the case of a \emph{Serre} fibration this is not an obstacle and the proof remains essentially unchanged.
    \begin{lemma}\label{lemma:localserre}Let $p:U\to Y$ be a continuous function between topological spaces such that for every $n\in \mathbb{N}$ there exists an open cover $\mathfrak{U}_n=\{Y_\alpha\}_{\alpha \in A}$ of $Y$ with the property that for every $\alpha\in A$ the map
    \be p|_{p^{-1}(Y_\alpha)}:p^{-1}(Y_\alpha)\to Y_\alpha\ee
    has the homotopy lifting property with respect to all $n$-dimensional CW-complexes. Then $p$ is a Serre fibration.
    \end{lemma}
    \begin{proof}

    Recall that in order for $p$ to be a Serre fibration, it is enough to check that it has the homotopy lifting property with respect to all cubes.
    Let $\tilde H:I^n\times I\to Y$ be a homotopy, $\tilde h_t=\tilde H(\cdot, t)$ and $h_0:I^n\to U$ be a lift; consider the open cover $\mathfrak{U}_n$ and a subdivision of $I^n$ into small cubes $\{C_\beta\}_{\beta \in B}$ and of $I$ into small intervals $\{I_j\}_{j=1, \ldots, N}$ with the property that for every $(\beta, j)\in B\times \{1, \ldots, N\}$ there exists $\alpha \in A$ such that:
    \be\label{eq:cont} \tilde H(C_\beta \times I_j)\subset Y_\alpha.\ee
    We can assume by induction that $h_t$ has been constructed over $\partial C_\beta$ for every $C_\beta$. To extend $h_t$ over the all cube $C_\beta$ we use \eqref{eq:cont} and notice that it means that we are given a homotopy $\tilde h_t|_{C_\beta}:C_\beta\to Y_\alpha$ with a lift:
    \be\label{eq:constr} h_t|_{\partial C_\beta}:\partial C_\beta\to p^{-1}(Y_\alpha).\ee
    By assumption $p|_{p^{-1}(Y_\alpha)}$ has the homotopy lifting property with respect to all $n$-dimensional CW-complexes. This is equivalent to the fact that $p|_{p^{-1}(Y_\alpha)}$ has the homotopy lifting property with respect to all $n$-dimensional CW-pairs (see \cite[Section 4.2]{Hatcher}); that means exactly that we can lift the homotopy $\tilde h_t|_{C_\beta}:C_\beta\to Y_\alpha$ with the constraint \eqref{eq:constr}.
    \end{proof}

\begin{thm}[Serre fibration property]\label{thm:serre}Assume all singular curves with $J\leq E_2$ are soft. If there are no normal geodesics in $\mathcal{U}(y)$ with Energy $E_1< J\leq E_2$, then for every $\epsilon>0$ sufficiently small the restriction of the Energy to $\mathcal{U}(y)_{E_1+\epsilon}^{E_2}$ is a Serre fibration.
    \end{thm}

      \begin{proof}
We will prove that for every $n\in \mathbb{N}$ there exists $\delta=\delta(n)$ such that for every point $s\in [E_1+\epsilon, E_2]$, denoting by:
\be I(s)=[E_1+\epsilon, E_2]\cap (s-\delta, s+\delta)\quad \textrm{and}\quad \mathcal{U}(y)_{I(s)}=\mathcal{U}(y)\cap J^{-1}(I(s)),\ee
then $J|_{J^{-1}(I(s))\cap \mathcal{U}(y)}$ has the homotopy lifting property with respect to all $n$-dimensional CW-complexes (or, equivalently with respect to all $n$-dimensional disks). The result will then follow from Lemma \ref{lemma:localserre}.

We denote by $\psi_t(\cdot)$ the deformation $\Theta(t,\cdot)$ coming from Proposition \ref{prop:defuniform}; it has the property that for every $u\in \mathcal{W}(\mathcal{B}^E)\cap \{E_1+\epsilon, J\leq E_2\}$ and $t<\epsilon'$:
\be \psi_t(u)\in \mathcal{U}(y)^E\backslash \mathcal{V}(A^{\epsilon})\quad \textrm{and}\quad J(\psi_t(u))=J(u)-t.\ee
Recalling the definition of $\mathcal{W}(\mathcal{B}^E)$ given in \eqref{eq:W}, we define the open set:
\be \mathcal{ K}(\mathcal{B}^{E})=\bigcup_{u\in \mathcal{B}^E}B(u, \beta/4).\ee
Let $r_2$ be given by Theorem \ref{thm:implicit} and define the number:
\be \mu=\min\left\{\frac{\beta}{8}, \frac{r_2}{3}\right\}.\ee

By Proposition \ref{prop:defuniform}, if $u\in \mathcal{K}(\mathcal{B}^E)$, and  $t<\min\{\epsilon',\frac{\mu}{n}\frac{1}{\eta}\}$, then:
\be\label{eq:disp1} \|u-\psi_t(u)\|\leq \frac{\mu}{n}.\ee

Moreover by the equicontinuity (at zero) of the family of functions $\{\sigma_{u_0}\}$ from Theorem \ref{thm:implicit}, there exists $c$ such that if $\|w-\varphi(u_0)\|+\|u-u_0\|\leq c$ then:
\be\label{eq:disp2} \|u+\sigma_{u_0}(w,u)-u_0\|\leq \frac{\mu}{n}.\ee

We define accordingly:
\be \delta=\delta(n)=\min\left\{\frac{\epsilon'}{2},\frac{\mu}{2n}\frac{1}{\eta}, \frac{c}{2}\right\}.\ee

Consider then a map $h_0:D^n\to \mathcal{U}(y)_{I(s)}$ lifting the homotopy $\tilde h_t:D^n\to I(s)$ at time $t=0.$ Endow $D^n$ with a CW-complex structure such that each cell is either entirely contained in $\mathcal{K}(\mathcal{B}^E)$ or entirely contained in $h_0^{-1}(B_{\frac{\mu}{n}}(h(x)))$ for some $x\in D^n$.

We lift the homotopy inductively on the skeleta of $D^n$, starting from the zero skeleton. If $x\in D^n$ is a point in the zero skeleton such that $h(x)\in\mathcal{K}( \mathcal{B}^E)$, then we define the homotopy $h_t|_{\{x\}}$ by:
\be
h_t(x)=\psi_{\tilde h_0(x)-\tilde h_t(x)}(h_0(x)),
\ee
in such a way that $J(h_t(x))=J(h_0(x))-\tilde h_0(x)+\tilde h_t(x)=\tilde h_t(x).$

If otherwise $x\notin\mathcal{K}( \mathcal{B}^E)$, then $x\in  B_{\frac{\mu}{n}}(u_0(x))$ for some $u_0=u_0(x)$. We consider the corresponding function $\sigma=\sigma_{u_0(x)}:B(0, r_1)\times B(u_0(x), r_2)\to L^2$ given by Theorem \ref{thm:implicit}.  Then $h_t|_{\{x\}}$ is defined as:
\be h_t(x)=h_0(x)+\sigma_{u_0(x)}(\underbrace{( \tilde h_t(x), y), h_0(x)}_{\in \R\times\R^m\times L^2 })\ee
which lifts $\tilde h_t(x)$ by Theorem \eqref{thm:implicit}.

Notice that, because of \eqref{eq:disp1} and \eqref{eq:disp2}, during the homotopy the point $h_t(x)$ has been moved from its original location at most at a distance:
\be\label{eq:dis} \|h_0(x)-h_t(x)\|\leq \frac{\mu}{n} .\ee

Assume now that the homotopy has been lifted to the $(k-1)$-skeleton of $D^n$ (with the CW-complex structure defined before). Composing with the characteristic map $\phi:D^k\to D^n$ of a cell, reduces to the case when we have to extend to the whole disk $D^k$ a homotopy which has been defined already on $\partial D^k$.

If $h_0(D^k)$ is entirely contained in $\mathcal{K}(\mathcal{B}^E)$, we simply define the homotopy using the flow as above:
\be h_{t}(x)=\psi_{\tilde h_0(x)-\tilde h_t(x)}(h_0(x)), \quad x\in D^k.\ee
This homotopy glues on the boundary $\partial D^k$, which by \eqref{eq:dis} is also entirely contained in $\mathcal{K}(\mathcal{B}^E)$ and for which the homotopy was defined by the flow.

Otherwise there exists $u_0$ such that $h_0(D^k)\subset B_{\frac{\mu}{n}}(u_0)$. We need to extend to the all disk a homotopy that has been defined already on $\partial D^k$; notice that since at each previous inductive step the homotopy moved the points from their original location at a distance at most $\frac{\mu}{n}$, then:
\be h_t(\partial D^k)\subset  B_{\frac{k \mu}{n}}(u_0)\subset  B_{r_2}(u_0) \quad \forall t\in I.\ee
In other words, denoting by $H:D^n\times I\to \mathcal{U}(y)$ the partially defined homotopy, we have:
\be H(D^{k}\times I, D^{k}\times \{0\}\cup \partial D^k\times I)\subset B_{r_2}(u_0).\ee
The pairs $(D^{k}\times I, D^{k}\times \{0\}\cup \partial D^k\times I)$ and  $(D^k\times I, D^{k}\times \{0\})$ are homeomorphic and the extension problem is equivalent to just the homotopy lifting property for a map from the $k$-dimensional disk to $I(s)$ with a time-zero lift all in $B_{r_2}(u_0)$.
For such a map $h_0:D^k\to B_{r_2}(u_0)$ we define again the lifting homotopy as:
\be  h_t(x)=h_0+\sigma_{u_0(x)}((\tilde h_t(x),y), h_0(x)).\ee
This concludes the proof.
\end{proof}
  \subsection{The deformation Lemma}\label{sec:deformation}
   \begin{thm}[subriemannian Deformation Lemma]\label{thm:deformation}
    	Assume that  all singular curves with $J\leq E_2$ are soft and that there are no normal geodesics in $\Omega(y)$ with Energy between $E_1$ and $E_2$.
	Then for every $\epsilon>0$, every compact manifold $X$ and any continuous map $h:X\to\Omega(y)^{E_2}$ there exists an homotopy $h_t:X\to \Omega(y)^{E_2}$ such that $h_0=h$ and $h_1(X)\subset \Omega(y)^{E_1+\epsilon}$.
    \end{thm}
    \begin{proof}
    Let us first show how to reduce the proof for horizontal curves to the proof of the exact same statement for the global chart. Notice first that, as we have shown in Section  \ref{subsect:genericprop}, if a singular curve $\gamma$ is soft, then the same is true for any control whose associated trajectory is $\gamma$; moreover if $u$ is a singular control with $J(u)\leq E_2$ which is not soft, then the corresponding trajectory $\gamma_u$ is also singular, it is not soft and its energy satisfies $J(\gamma_u)\leq E_2$. Also if $u$ is a normal control, then $A(u)$ is a normal geodesic. This can be seen as follows: being locally length minimizing, $A(u)$ can be either the projection of a normal or an abnormal extremal. If it were the projection of an abnormal extremal, it would be a singular curve, hence of corank one, and strictly abnormal, contradicting the existence of a normal extremal lift for $u$.
In particular the hypothesis at the level of horizontal curves imply the same hypothesis for the set of controls.

    Given $h$, we can consider the function:
    \be \overline{h}:X\to\mathcal{U}, \quad \overline{h}(\theta)=\mu(h(\theta))\ee
    (recall that $\mu:\Omega \to \mathcal{U}$ denotes the minimal control). Then, because $J(\mu(\gamma))=J(\gamma)$, we have $\overline{h}(X)\subset \mathcal{U}(y)^{E_2}$.

   If we can find a homotopy $\overline{h}_t:X\to \mathcal{U}(y)^{E_2}$ with the property that $\overline{h}_1(X)\subset \mathcal{U}(y)^{E_1+\epsilon}$, then the function:
   \be h_t=A\circ \overline h_t:X\to \Omega(y)\ee
   defines the desired homotopy. In fact $h_0(\theta)=A(\overline h_0(\theta))=A(\mu(h(\theta)))=h(\theta)$ for every $\theta\in X$; moreover since $J(A(\overline h_t(\theta)))\leq J(\overline h_t(\theta))$ then we also have $h_t(X)\subset \Omega(y)^{E_2}$ and $h_1(X)\subset \Omega(y)^{E_1+\epsilon}.$

  This reduces to prove that the same statement holds true if $\Omega(y)$ is replaced with $\mathcal{U}(y)$.
Consider the two open sets of $V_1=h^{-1}(\{J<E_1+\epsilon\})$ and $V_2=h^{-1}(\{J> E_1+\epsilon/2\})$. Let $\{\rho_1, \rho_2\}$ be a smooth partition of unity subordinated to the open cover $\{V_1, V_2\}$ of $X$. Then $\rho_2|_{h^{-1}(\{J\geq E_1+\epsilon\})}\equiv 1$ and $\rho_2|_{h^{-1}(\{J\leq E_1+\epsilon/2\})}\equiv 0$. Let $c\in (E_1+\epsilon/2, E_1+\epsilon)$ be a regular value of $\rho_2$ and consider the smooth submanifold:
    \be M=\{\rho_2\geq c\}\quad\textrm{with} \quad \partial M=\{\rho_2=c\}.\ee
    The pair $(M, \partial M)$ is a CW-complex pair and $h(M)\subset  \{J\geq c\}.$ Then by Theorem \ref{thm:serre} there exists a homotopy
    \be\tilde H:M\times I\to \Omega(y)\cap \{c\leq J\leq E_2\}\ee
     such that $\tilde H(\cdot, 0)=h|_M$, $\tilde h_1(M)\subset \{J\leq \epsilon\}$ and $\tilde H(\cdot, t)|_{\partial M}\equiv h|_{\partial M}.$
     The desired homotopy $H:X\times I\to \mathcal{U}(y)^{E_2}$ is defined by:
     \be H(x, t)=\left\{\begin{array}{cc}h(x)&\overline x\in X\backslash M\\
	&\\
	\tilde H(x,t)& x\in M\end{array}\right.\ee
       \end{proof}
	
	\section{Applications}\label{sec:applications}
\subsection{A subriemannian Minimax principle}
In this section we prove a subriemannian version of the classical Minimax principle for variational problems. If $X$ is a compact manifold and two continuous maps $f,g:X\to \Omega(y)$ are homotopic we will write $f\sim g$.
\begin{thm}[subriemannian Minimax principle]\label{thm:minmax} Let $X$ be a compact manifold and $f:X\to \Omega(y)$ be a continuous map. Consider:
\be c=\inf_{g\sim f}\sup_{\theta \in X} J(g(\theta)).\ee
Assume that $c>0$ and that there exists $\delta>0$ such that all singular curves with energy $J\leq c+\delta$ are soft.
Then for every $\delta>\epsilon>0$ there exists a normal geodesic $\gamma_\epsilon$ such that:
\be c-\epsilon\leq J(\gamma_\epsilon)\leq c+\epsilon.\ee

\end{thm}

\begin{proof}Assume that the claim is false. Then there exists $\delta>\epsilon>0$ such that any curve in $\Omega(y)^{c+\epsilon}_{c-\epsilon}$ is either regular or a soft abnormal. 	
Then let $g\sim f$ such that $\sup_{\theta \in X} J(g(\theta))\leq c+\epsilon$. By Theorem \ref{thm:deformation} the map $g$ is homotopic to a map $g':X\to \Omega(y)^{c-\epsilon/2}$, which contradicts the definition of $c.$
\end{proof}

Note that if $x$ (the starting point) is different from $y$, then the Minimax value $c$ given by the previous theorem is automatically different from zero. In the case $x=y$ we have the following result (no assumption on the singular curves is required).
\begin{prop}Assume $x=y$. Then there exists $\epsilon>0$ such that any map $f:X\to \Omega(x)$ satisfying $\sup_{\theta\in X}J(f(\theta))\leq \epsilon$ is homotopic to a constant map.
\end{prop}
\begin{proof}By \cite[Corollary 7]{BoarottoLerario} the space $\Omega(x)$ has the homotopy type of a CW-complex, and in particular any point in it has a contractible neighborhood. Consider the constant curve $\gamma(y)\equiv x$, and a neighborhood $U_{\gamma}\subset \Omega(x)$ which is contractible in $\Omega(x)$. Since the family $\{J< t\}_{t\in \R}$ is a local basis for $\Omega(x)$ at the constant curve $\gamma$, then there exists $\epsilon$ such that $\{J\leq\epsilon\}\subset U_{\gamma}$. As a consequence, if $\textrm{im} (f)\subset \{J\leq \epsilon\}$, then $f$ is homotopic to a constant map.
\end{proof}

\subsection{Serre's theorem and another deformation lemma}

We finish this section with a proof of a subriemannian version of Serre's Theorem, providing the existence of infinitely many geodesics between any two points $x$ and $y$ on a compact subriemannian manifold. The case when $y$ is a regular value for the endpoint map centered at $x$ and the case of a contact manifolds are proved in \cite{BoarottoLerario}. We will need the following variation of the deformation lemma.

\begin{lemma}Assume all singular curves with $J\leq E$ are soft. Let $0<s<E$ and $\mathcal{M}$ be a neighborhood of $\mathcal{C}\cap\{J=s\}$. Then for every $n\in \mathbb{N}$ there exists $\epsilon=\epsilon(n)$ such that for every $n$-dimensional simplicial complex $X$ and any continuous map:
\be f:X\to\Omega(y)^{s+\epsilon}\backslash \mathcal{M}\ee
there exists a homotopy of maps $f_t:X\to \Omega(y)^{s+\epsilon}\backslash \mathcal{M}$, $t\in [0,1]$, such that $f_0=f$ and:
\be f_1:X\to \Omega^{s-\epsilon}.\ee
\end{lemma}

\begin{proof}
First, arguing as in the proof of Theorem \ref{thm:deformation}, we reduce to prove the statement for $\mathcal{U}$ instead of $\Omega$, thus working with $\mathcal{N}=A^{-1}(\mathcal{M})$ instead of $\mathcal{M}$.

We claim that the statement follows from the following fact (whose proof we postpone): there exists $\epsilon>0$ such that for every $n$-dimensional simplicial pair $(Y, Z)$, $Z\subset Y$ and any continuous map $f:Y\to\left( \mathcal{U}(y)\backslash\mathcal{M}\right)\cap\{s-2\epsilon\leq J\leq s+\epsilon\}$ such that $f(Z)\subset \{J\leq s-\epsilon\}$ we can find a homotopy $f_t:Y\to \left( \mathcal{U}(y)\backslash\mathcal{M}\right)\cap\{s-2\epsilon\leq J\leq s+\epsilon\}$ with $f_1(Y)\subset \{J\leq s-\epsilon\}$ and $f_t|_{Z}\equiv f.$
To see that this implies the statement consider $c\in (s-2\epsilon, s-\epsilon)$ and the two open sets $V_1=f^{-1}(\{J<s-\epsilon\})$ and $V_2=f^{-1}(\{J>c\}).$ Since $X$ is a simplicial complex, there exist subcomplexes $Y_1, Y_2\subset X$ such that $X=Y_1\cup Y_2$ and $Y_1\subset V_1$, $Y_2\subset V_2$. Applying the claim to the map $f|_{Y_2}$ and the pair $(Y_2, Y_2\cap Y_1)$ provides a homotopy $\tilde f_t:Y_2\to  \left( \mathcal{U}(y)\backslash\mathcal{M}\right)\cap\{s-2\epsilon\leq J\leq s+\epsilon\}$ which is stationary on $Y_1\cap Y_2$ and which consequently glues to a homotopy $f_t:X\to  \left( \mathcal{U}(y)\backslash\mathcal{M}\right)\cap\{s-2\epsilon\leq J\leq s+\epsilon\}$ which is  just $f$ on $Y_2$ and which satisfies the needed requirements.

It remains to prove the claim. Arguing as in Lemma \ref{lemma:localserre}, we see that instead of proving it for any simplicial pair, it is enough to do it for a map from the disk $D^n$:
\be f:D^n \to\left( \mathcal{U}(y)\backslash\mathcal{M}\right)\cap\{s-2\epsilon\leq J\leq s+\epsilon\}\ee
($\epsilon$ will be chosen below).

First we use Corollary \ref{cor:normgeod}, which provides us with $0<\eta'<\eta$ such that for every $\epsilon'<\epsilon''<\eta'<\eta$ we have a deforming map:
\be \psi_t:\mathcal{W}(\mathcal{B}^E)\to \mathcal{U}(y)\backslash \mathcal{V}(\mathcal{A}^E)\ee
with the properties: $\psi_0\equiv \textrm{id}$, $\psi_t|_{\{J\leq s-\epsilon''\}\cup\{J\geq s+\epsilon''\}}\equiv \textrm{id}$ and:
\be\psi_{\overline{t}}\left((\mathcal{W}(\mathcal{B}^E)\backslash \mathcal{N})\cap\{J\leq s+\epsilon'\}\right)\subset \mathcal{U}^E\backslash \mathcal{V}(\mathcal{A}^E)\cap \{J\leq s-\epsilon'\}\ee
(here $\overline t=\eta\beta/2$).
We define $\tilde\psi_t=\psi_{t/\overline t}$, $t\in [0,1]$ so that we have completely deformed at time $t=1;$ moreover (using the notation of Proposition \ref{prop:def}) if $t<\frac{\mu \overline t \eta}{n}$ we also have:
\be \|\psi_t(u)-u\|\leq \int_{0}^{\frac{t}{\overline t}}\|Y(\psi_\tau^Y(u))\|d\tau\leq \frac{t}{\overline t \eta}\leq \frac{\mu}{n}.\ee
As before, by the equicontinuity (at zero) of the family of functions $\{\sigma_{u_0}\}$ from Theorem \ref{thm:implicit}, there exists $c$ such that if $\|w-\varphi(u_0)\|+\|u-u_0\|\leq c$ then:
\be\label{eq:disp2} \|u+\sigma_{u_0}(w,u)-u_0\|\leq \frac{\mu}{n}.\ee

We define accordingly:
\be \epsilon=\epsilon(n)=\min\left\{\frac{\epsilon'}{2},\frac{\mu}{2n}\frac{1}{\eta}, \frac{c}{2}, \frac{r_1}{2}\right\}.\ee
The proof now proceeds exactly as the proof of Theorem \ref{thm:serre}, with the choice $h_0=f$ and $\tilde{h}_t$ defined by:
\be \tilde{h}_t(z)=J(f(z))(1-t)+t(s-\epsilon), \quad z\in D^n.\ee

\end{proof}

\begin{cor}[Serre's Theorem]\label{cor:serre}Let $x,y$ be any two point on a \emph{compact} subriemannian manifold whose singular curves are all soft. Then there are infinitely many normal geodesics joining $x$ and $y$.
\end{cor}
\begin{proof}The scheme of the proof is identical to the classical proof (see for example \cite[Section 3.2]{Chang}), except that here we have to work with the (possibly singular) space $\Omega(y)$ of horizontal curves joining $x$ and $y$ and with the functional $J:\Omega(y)\to \R$. The main difference is to replace the classical deformation lemma with the one above. We only give a sketch, leaving the details to the reader.

In the case the fundamental group of $M$ is infinite, it is enough to apply a Minimax procedure on each homotopy class (as it is done in \cite{LerarioMondino}); otherwise one passes to the universal cover, which is locally isometric to the manifold itself, and proves the statement for the case $M$ is compact and simply connected.

The homotopy of $\Omega(y)$ is the same as the homotopy of the ordinary path-space (in fact the inclusion is a homotopy equivalence, see \cite{BoarottoLerario}), and under the assumption that $M$ is compact and simply connected the cup-length of $\Omega(y)$ is infinite \cite{Serre}.
The theorem follows then by applying the same argument as in the proof of \cite[Lemma 3.1]{Chang}, with the following modification. Using the notation of \cite[Lemma 3.1, pag. 106]{Chang}, the simplicial set $|\tau_1|$ is deformed below the level $c_2-\epsilon$ using the previous Lemma.
\end{proof}

\end{document}